\documentclass[12pt]{article}
\usepackage[utf8]{inputenc}
\usepackage[normalem]{ulem}
\usepackage[blocks]{authblk}
\useunder{\uline}{\ul}{}
\usepackage{ dsfont }
\usepackage{amsthm}
\usepackage{amsmath}
\usepackage{ amssymb }
\usepackage{tikz}
\usetikzlibrary{decorations.pathreplacing,calligraphy}
\usepackage{caption}
\usepackage{subcaption}
\usepackage{url}
\usepackage{graphicx}
\usepackage{ upgreek }
\usepackage{mathrsfs}
\usepackage{dirtytalk}
\usepackage{hyperref}
\usepackage{geometry}
\usepackage{multicol}

\newtheorem{theorem}{Theorem}[section]
\newtheorem{remark}[theorem]{Remark}
\newtheorem{definition}[theorem]{Definition}
\newtheorem{proposition}[theorem]{Proposition}
\newtheorem{lemma}[theorem]{Lemma}
\newtheorem{example}[theorem]{Example}

\begin{document}
\title{Non-equilibrium steady states with a spatial Markov structure}
\author[1]{Frank Redig
\thanks{Email: \texttt{F.H.J.Redig@tudelft.nl}}
}
\author[2]{Berend van Tol
\thanks{Email: \texttt{B.T.vanTol@tudelft.nl}}
}

\affil[1,2]{Institute of Applied Mathematics, Delft University of Technology, Delft,
The Netherlands}
\date{\today}

\maketitle
\begin{abstract}
We investigate the structure of non-equilibrium steady states (NESS) for a class of exactly solvable models in the setting of a chain with left and right reservoirs.
Inspired by recent results on the harmonic model \cite{carinci2023large},
\cite{davide2}, we focus on models in which the NESS is a mixture of equilibrium product measures, and where the probability measure which describes the mixture has a spatial Markovian property. We completely characterize the structure of such mixture measures, and show that under natural scaling and translation invariance properties, the only possible mixture measures are coinciding with the Dirichlet process found in \cite{carinci2023large} in the context of the harmonic model.
\end{abstract}
\section{Introduction}
The structure of the stationary state of stochastic interacting particle systems driven out of equilibrium, for instance via boundary driving, is a challenging problem, and very few exactly solvable models are available.
A famous example of an exactly solvable model in this context is the boundary driven exclusion process, solvable via the matrix ansatz solution, see \cite{derrida}, \cite{derrida2001free}.

Recently, new explicit representations of non-equilibrium steady states (NESS) were discovered in the context of the boundary driven KMP model 
\cite{KMP}, \cite{bertini2007stochastic}, \cite{davide1}, and in the context of the boundary driven harmonic model \cite{carinci2023large},
\cite{davide2}.
For a general class of mass transport models, including both the KMP and the harmonic model on general graphs, such representations were found via intertwining in 
\cite{giardina2024intertwining}.
The integrable structure of the harmonic model was discovered in \cite{frassek2022exact}, and explicit formulas for the factorial moments in the  non-equilibrium stationary state are given.
These explicit representations of the NESS allow to obtain very detailed information about large deviations, long-range correlations and other properties of NESS that are believed to hold for a large class of models. 

\subsection{The non-equilibrium steady state as a mixture}
For all the models mentioned above (KMP, harmonic), the equilibrium distributions are product. E.g.\ for the KMP model the equilibrium measures are products of exponentials with identical (i.e., site independent) parameters. 
An important development is the discovery that for all these models the NESS can be written as a stochastic mixture of product measures with equilibrium marginals. E.g.\ for the KMP model on a chain with $n$ sites coupled to reservoirs at both ends, with left reservoir parameter $\theta_L$ and right reservoir parameter $\theta_R>\theta_L$, the NESS reads
\begin{equation}\label{kmpness}
\mu_{\theta_L,\theta_R}^{(n)}[d\eta_1,\ldots, d\eta_n]
= \int \left(\otimes_{i=1}^n \nu_{\theta_i}[d\eta_i]\right)
\Lambda^n_{\theta_L,\theta_R}(d\theta_1,\ldots d\theta_n).
\end{equation}
Here $\nu_{\theta_i}$ is an exponential distribution with parameter (i.e., expectation) $\theta_i$, and
$\Lambda^n_{\theta_L,\theta_R}$, the so-called mixture measure, is a probability measure on $[\theta_L,\theta_R]^n$ concentrating on ordered
$n$-tuples
$\theta_L\leq \theta_1\leq\ldots\leq \theta_n\leq \theta_R$.
This mixture measure describes the distribution of the ``hidden parameters'' $(\theta_1,\ldots,\theta_n)$.
We then recover the equilibrium product measures of the KMP model when $\theta_L=\theta_R=\theta^*$ because in that case the mixture measure $\Lambda^n_{\theta_L,\theta_R}$ becomes
a Dirac measure concentrating on $\theta^*$.

Moreover, the mixture measure $\Lambda^n_{\theta_L,\theta_R}$ in \eqref{kmpness}
is in turn the stationary measure of a Markovian dynamics on the space of parameters $(\theta_1,\ldots,\theta_n)\in [\theta_L,\theta_R]^n$.
This dynamics of the parameters is then called the ``hidden temperature'' model \cite{davide1} or in more general setting, the ``hidden parameter'' model \cite{giardina2024intertwining}.

In the context of the (discrete or continuous) harmonic model, a representation as in \eqref{kmpness} exists, and moreover, in that context, the measure
$\Lambda^n_{\theta_L,\theta_R}$ is available in closed form and given by the joint distribution of the order statistics of $n$ independent uniforms on the interval $[\theta_L,\theta_R]$ (here we assume $\theta_L<\theta_R$). 
In \cite{carinci2023large}, a class of generalized harmonic models is introduced, where also the mixture measure $\Lambda^n_{\theta_L,\theta_R}$ is explicit and given by an ordered Dirichlet distribution.
Other models in which the NESS can be written as a stochastic mixture of product measures include the boundary-driven symmetric exclusion process \cite{floreani2023non} and the boundary-driven inclusion process \cite{giardina2024intertwining}.
In both these models (as well as in the KMP model), the mixture measure cannot be obtained explicitly as it is the case for the harmonic models (on the chain), but it can only be characterized implicitly as the stationary measure of the hidden parameter model.
\subsection{The use of mixture representations of a NESS }

A representation of the NESS of type \eqref{kmpness}
with explicit knowlegde of the mixture measure 
is very useful and allows to derive many macroscopic properties 
about the density profile in the NESS.
More precisely, the explicit knowledge of the mixture measure $\Lambda^n_{\theta_L,\theta_R}(d\theta_1,\ldots d\theta_n)$ allows to  prove the large deviation principle for the density profile
$\frac1n \sum_{i=1}^n \eta_i\delta_{i/n}$ under the NESS, as already outlined in \cite{bertini2007stochastic}. 
Conditional on a realization of the hidden parameters $(\theta_1,\ldots,\theta_n)$, the measure $\mu_{\theta_L,\theta_R}^{(n)}$ is a product measure
$\otimes_{i=1}^n \nu_{\theta_i}$ for which the computation of the rate function for
the large deviations of the density profile is easy, and yields a function of the empirical profile of the parameters
$\frac1n\sum_{i=1}^n \theta_i \delta_{i/n}$. So once one can obtain the large deviation principle of
$\frac1n\sum_{i=1}^n \theta_i \delta_{i/n}$ under the 
measure $\Lambda^n_{\theta_L,\theta_R}$, one obtains a variational expression for the rate function controlling the large deviations of the density profile. This variational expression naturally leads to the non-locality of the large deviation rate function,
a property which is expected to be generic for non-equilibrium steady states. The representation \eqref{kmpness} also naturally leads to explicit formulas for correlation functions and cumulants, relating them to correlation functions of order statistics, which are known explicitly. In particular, we can recover the long-range character of the covariance in the NESS, and we can also obtain natural scaling properties of cumulants. 

As an additional by-product of a representation of the form \eqref{kmpness}, one can prove local equilibrium and have control on the (mesoscopic) deviations from local equilibrium. 

In that sense, it is very relevant to obtain models in which the NESS can be written in the form \eqref{kmpness} with an explicit form for the mixture measure $\Lambda^n_{\theta_L,\theta_R}$.

\subsection{The spatial Markov property of the hidden parameter model}
In \cite{giardina2024intertwining} we obtain a probabilistic understanding of the fact that for the harmonic models the mixture measure can be obtained explicitly. We show that under the measure
$\Lambda^n_{\theta_L,\theta_R}$, the conditional distribution of $\theta_i$ given all $\theta_j, j\not= i$, is uniform on $[\theta_{i-1}, \theta_{i+1}]$, which is exactly the Markov structure of the joint distribution of order statistics of independent uniforms. 

The spatial Markov structure of the measure
$\Lambda^n_{\theta_L,\theta_R}$ implies that one can reconstruct $\Lambda^n_{\theta_L,\theta_R}$ from
$\Lambda^1_{\theta_L,\theta_R}$.  More precisely, knowing the invariant measure for the hidden parameter model for a single-site system with left and right reservoirs is sufficient to obtain full knowledge of $\Lambda^n_{\theta_L,\theta_R}$ for all system sizes $n\in\mathbb{N}$. Finding this invariant measure for the system with a single site is usually an easy problem. Then, as a consequence of the structure of the NESS, one can find $\mu_{\theta_L,\theta_R}^{(n)}[d\eta_1,\ldots, d\eta_n]$
for all $n\in\mathbb{N}$.

As we showed in \cite{carinci2023large}, this spatial Markov structure of the measure
$\Lambda^n_{\theta_L,\theta_R}$ also leads naturally to the additivity principle for the pressure and the corresponding large deviation rate function (cf. \cite{bodineau2004current}). Therefore,  we can think of the spatial Markov structure 
of the mixture measure $\Lambda^n_{\theta_L,\theta_R}$ as a microscopic counterpart of the additivity principle \cite{bodineau2004current}.

\subsection{Content and organization of the paper}

In this paper, we focus
on the characterization and the precise consequences of the Markov structure of the mixture measure
$\Lambda^n_{\theta_L,\theta_R}$.

As a first result, we prove that in such a setting
the measures $\Lambda^n_{\theta_L,\theta_R}$
are completely determined by the system consisting of a single site with left and right reservoirs, i.e., the measures
$\Lambda^1_{\theta_L,\theta_R}$. As a second result, we characterize those single-site systems which are ``extendable'', i.e.,  give rise to a Markov family $\{\Lambda^n_{\theta_L,\theta_R}, n\in\mathbb{N}\}$.
Finally, as a third result, under natural assumptions of scaling and translation invariance, we prove that the only possible measures
$\Lambda^n_{\theta_L,\theta_R}$ are the ordered Dirichlet distributions which appeared in the generalized harmonic models \cite{carinci2023large}.

The measures $\Lambda^n_{\theta_L,\theta_R}$ are concentrated on the set of ordered $n$-tuples
$\theta_L<\theta_1<\theta_2<\ldots<\theta_n<\theta_R$. Because of this ordering restriction, the standard theory of Markov specifications from 
\cite{georgii} chapter 10 and 11, or the Hammersley-Clifford theorem \cite{Hamm} is not directly applicable as both assume non-null specifications. We can think of the measures
$\Lambda^n_{\theta_L,\theta_R}$
as the joint distribution of a Markov bridge, i.e., a joint distribution of the form
\[
X^{(n)}_1=\cdot,\ldots, X^{(n)}_{n}=\cdot|X^{(n)}_0=\theta_L, X^{(n)}_{n+1}=\theta_R
\]
where $X^{(n)}_1=\cdot,\ldots, X^{(n)}_{n}$ are the first $n$ steps of a Markov process. However, because of the ordering restriction, 
this Markov bridge depends on $n$ and therefore cannot be thought of as the conditioning of an infinite-volume Gibbs measure with nearest neighbor interaction.

The rest of our paper is organized as follows. In section 2 we define the two-sided Markov property and a notion of symmetry for the measures $\Lambda^n_{\theta_L,\theta_R}$. In section 3 we study their support properties. In section 4 we provide a way to reconstruct the measure $\Lambda^n_{\theta_L,\theta_R}$ corresponding to $n$ sites from the measure corresponding to a single site. Finally, in section 5, we find two properties concerning the shift and scale invariance of the measures to characterize the Dirichlet distributions.

\section{Structure of the $n$-site densities}
In this section we first introduce the necessary notation and give a formal definition of the Markovian structure of the mixture measure $\Lambda^n_{a,b}$. We also define a notion of symmetry which we will assume throughout this text. Next we prove that the Markov property implies a product structure for the $n$-site density and subsequently investigate consequences of the symmetry property.\\

\noindent
Let $a,b \in \mathbb{R}^+$. We consider a family of probability density functions $\{\Lambda_{a,b}^n\}_{n \geq 1}$ corresponding to random vectors taking values in $(I_{a,b})^n$, where $I_{a,b} = (a \wedge b, a \vee b)$. That is, for fixed $a$, $b$ and $n$ we have random variables $\Theta_{1,n},...,\Theta_{n,n}$ with joint density $\Lambda^n_{a,b}$:

\begin{equation}
    \mathbb{P}_n((\Theta_{1,n} , ..., \Theta_{n,n}) \in A) = \int_A \Lambda_{a,b}^n d\lambda_n,
\end{equation}

\noindent
for each Borel set $A \in \mathscr{B}((I_{a,b})^n)$ and $n \geq 1$ and where $\lambda_n$ denotes the Lebesgue measure on $\mathbb{R}^n$. We assume that, for all values of $a,b \in  \mathbb{R}^+$ and $n \geq 1$, we have a corresponding density $\Lambda_{a,b}^n$. This then defines the map $\Lambda^n: (a,b)\rightarrow \Lambda_{a,b}^n$ and we refer to this map as the $n$-site density. Moreover, we refer to $\Lambda_{a,b}^n$ as the  the $n$-site density with left parameter $a$ and right parameter $b$.\\

\noindent
Throughout this text we often use the densities of marginals and conditional densities of the random variables $\Theta_{i,n}$, $1\leq i \leq n$, which we denote as follows,

\begin{equation}
    \Lambda_{a,b}^{n,i}(\theta_i)
:= \int_{-\infty}^\infty ...\int_{-\infty}^\infty \Lambda_{a,b}^n(\theta_1, ...,\theta_i, ..., \theta_n) d\theta_1 ... d\theta_{i-1} d\theta_{i+1}...d\theta_n
\end{equation}

for the $i$-th marginal and, for $\theta_i \in \text{supp } (\Lambda_{a,b}^{n,i})$,

\begin{equation}
    \Lambda_{a,b}^n(\theta_1, ..., \theta_n|\Theta_i = \theta_i) := \frac{\Lambda_{a,b}^n(\theta_1, ...,\theta_i, ..., \theta_n)}{\Lambda_{a,b}^{n,i}(\theta_i)}.
\end{equation}

\noindent
We now define what it means for the family $\Lambda^n$ to be two-sided Markov. For notational simplicity we omit the right index $n$ in $\Theta_{1,n},...,\Theta_{n,n}$ whenever $n$ is fixed, i.e. we write $\Theta_1,...,\Theta_n$.


\begin{definition}
\label{def:sym_two_sided}
 The family $\{\Lambda^n\}_{n \geq 1}$ is called

 \begin{enumerate}
     \item  \textbf{Two-sided Markov} if the corresponding densities $\Lambda_{a,b}^n$, $n \geq 1$ and $a,b \in \mathbb{R}^+$,
satisfy the following two properties
     \begin{itemize}
         \item[(i)] Two-sided Markov property:

         \begin{equation}
         \label{eq:two-sided_Markov_def}
            \Lambda_{a,b}^n(\theta_1, ..., \theta_n|\Theta_j = \theta_j) = \Lambda_{a,\theta_j}^{j-1}(\theta_1,..., \theta_{j-1}) \cdot \Lambda_{\theta_j,b}^{n-j}(\theta_{j+1},...,\theta_n)
        \end{equation}

        for all $\theta_1,..., \theta_n \in I_{a,b}$. Here we define $\Lambda_{a,b}^0 := 1$.\\

        \item[(ii)] Support restriction:  $\Lambda_{a,b}^1$ has support $I_{a,b}$. Here we define the support of a function $f:X \rightarrow \mathbb{R}$ as $\text{\normalfont supp } (f) := \{x \in X: f(x) \neq 0\}$.
        
     \end{itemize}

    \item \textbf{Symmetric} if, for all $n \geq 1$, $a,b \in \mathbb{R}^+$ and $\theta_1,...,\theta_n \in I_{a,b}$

     \begin{equation}
     \label{eq:symm_def}
         \Lambda_{a,b}^n(\theta_1,...,\theta_n) = \Lambda_{b,a}^n(\theta_n,...,\theta_1).
     \end{equation}
 \end{enumerate}

\end{definition}

\begin{remark}
    In this paper we are interested in densities $\Lambda^n_{a,b}$ which concentrate on the set of ordered $n$-tuples $a<\theta_1<...<\theta_n<b$. This is guaranteed via the requirement on the support of $\Lambda^1$, as we will prove in proposition \ref{prop:supp_lambdas}.
\end{remark}

\begin{example}
\label{ex:order_stat_twosidedmarkov}
\normalfont
Consider the joint probability density of the order statistics of $n$ i.i.d.\ uniform random variables on $(a,b)$ with $a \leq b$. More precisely, for $n\geq1$ we consider $X_1,X_2,...,X_n \sim U(a,b)$ and denote the corresponding order statistics by $\Theta_{1} = X_{1:n}, \Theta_{2} = X_{2:n},...,\Theta_{n} = X_{n:n}$. The joint distribution of $\Theta_1,...,\Theta_n$ is given by  

\begin{equation}
    \Lambda_{a,b}^n(\theta_1,...,\theta_n) = n! \cdot |b-a|^{-n} \mathds{1}(a < \theta_1 < \theta_2 < ... < \theta_n < b).
\end{equation}

The marginals are given by 

\begin{equation}
\label{eq:marginals_order_stat}
    \Lambda_{a,b}^{n,i}(\theta) = \frac{n!}{(i-1)!(n-i)!}|b-a|^{-n} |b-\theta|^{n-i}|\theta-a|^{i-1}\mathds{1}(\theta \in I_{a,b}).
\end{equation}

The two-sided Markov property now follows from 

\begin{equation}
    \Lambda_{a,b}^n(\theta_1,...,\theta_n| \Theta_i=\theta_i) = \frac{\Lambda_{a,b}^n(\theta_1,...,\theta_n)}{\Lambda_{a,b}^{n,i}(\theta_i)} = \frac{(i-1)!(n-i)!}{|b-\theta_i|^{n-i}|\theta_i-a|^{i-1}} \mathds{1}(a < \theta_1 < \theta_2 < ... < \theta_n < b)
\end{equation}

and

\begin{equation}
    \Lambda_{a,\theta_i}^{i-1}(\theta_1,...,\theta_{i-1}) \cdot \Lambda_{\theta_i,b}^{n-i}(\theta_{i+1},...,\theta_n) = \frac{(i-1)!(n-i)!}{|b-\theta_i|^{n-i}|\theta_i-a|^{i-1}} \mathds{1}(a < \theta_1 < \theta_2 < ... < \theta_n < b).
\end{equation}

\noindent
We generalize the definition of the order statistics for the case $b \leq a$ in such a way that the family $\{\Lambda^n\}_{n\geq1}$ becomes symmetric. I.e. we define

\begin{equation}
    \Lambda_{a,b}^n(\theta_1,...,\theta_n) =
    \begin{cases}
        n! \cdot |b-a|^{-n} \mathds{1}(a < \theta_1 < \theta_2 < ... < \theta_n < b) \quad \text{if } a < b\\
        n! \cdot |b-a|^{-n} \mathds{1}(a > \theta_1 > \theta_2 > ... > \theta_n >b) \quad \text{if } a > b
    \end{cases}
\end{equation}

\noindent
One can show the two-sided Markov property for the case $a>b$ in the same way as is done above for $a<b$.

\end{example}

\subsection{Product form for two-sided Markov families}


\noindent
In this paper we aim to fully characterize the symmetric families which satisfy the two-sided Markov property. A first step in this direction is given by the proposition \ref{prop:characterization_prod} below. It states the equivalence between the two-sided Markov property and the density being a product of marginals in two distinct ways. These product forms can directly be derived by repeatedly applying the two-sided Markov property. The converse can be derived by combining the two product representations.

\begin{proposition}
\label{prop:characterization_prod}
 For notational convenience we write $\theta_0 :=a$ and $ \theta_{n+1}:=b$. The following statements are equivalent:\\

 \begin{enumerate}
     \item $\{\Lambda^n\}_{n \geq 1}$ is two-sided Markov.\\
     
     \item For all $n \geq 1$, $a,b \in \mathbb{R}^+$ and $\theta_1, ..., \theta_n \in I_{a,b}$ we have
    
    \begin{equation}
    \label{eq:end_markov}
    \begin{cases}
        \Lambda_{a,b}^n(\theta_1, ..., \theta_n|\Theta_n = \theta_n) = \Lambda_{a,\theta_n}^{n-1}(\theta_1, ..., \theta_{n-1})\\
        \Lambda_{a,b}^n(\theta_1, ..., \theta_n|\Theta_1 = \theta_1) = \Lambda_{\theta_1,b}^{n-1}(\theta_2, ..., \theta_{n}).
    \end{cases}
    \end{equation}\\

    \item For all $n \geq 1$, $a,b \in \mathbb{R}^+$ and $\theta_1, ..., \theta_n \in I_{a,b}$,

    \begin{equation}
    \label{eq:prod_forms}
         \Lambda_{a,b}^n(\theta_1, ..., \theta_n) = \prod_{i=0}^{n-1} \Lambda_{\theta_i,b}^{n-i,1}(\theta_{i+1}) = \prod_{i=0}^{n-1} \Lambda_{a,\theta_{n-i+1}}^{n-i,n-i}(\theta_{n-i}).
    \end{equation}
 \end{enumerate}
    
\end{proposition}

\begin{proof}
    1 implies 2: This follows immediately from the definition of the two-sided Markov property.\\

    \noindent
    2 implies 3: We use induction. \eqref{eq:prod_forms} obviously holds for $n=1$. Assume \eqref{eq:prod_forms} holds for $n=k$, then we have, using \eqref{eq:end_markov},

    \begin{equation}
        \Lambda_{a,b}^{k+1}(\theta_1, ..., \theta_{k+1}) = \Lambda_{a,b}^{k+1,k+1}(\theta_{k+1})\cdot\Lambda_{a,\theta_{k+1}}^{k}(\theta_1, ..., \theta_{k}) = \Lambda_{a,b}^{k+1,k+1}(\theta_{k+1}) \cdot \prod_{i=0}^{k-1} \Lambda_{a,\theta_{k-i+1}}^{k-i,k-i}(\theta_{k-i})
    \end{equation}

    and

    \begin{equation}
        \Lambda_{a,b}^{k+1}(\theta_1, ..., \theta_{k+1}) = \Lambda_{a,b}^{k+1,1}(\theta_{1})\cdot\Lambda_{\theta_1,b}^{k}(\theta_2, ..., \theta_{k+1}) = \Lambda_{a,b}^{k+1,1}(\theta_{1})\cdot \prod_{i=0}^{k} \Lambda_{\theta_i,b}^{k-i,1}(\theta_{i+1}).
    \end{equation}

    \noindent
    3 implies 1: We have the following two expressions for $\Lambda_{a,b}^n(\theta_1,...,\theta_n)$,

    \begin{align}
        \Lambda_{a,b}^n&(\theta_1,...,\theta_n)\nonumber\\ 
        &= [\Lambda_{a,b}^{n,1}(\theta_1) \cdot \Lambda_{\theta_1,b}^{n-1,1}(\theta_2)...\Lambda_{\theta_{j-1},b}^{n-j+1,1}(\theta_j)] \cdot [\Lambda_{\theta_j,b}^{n-j,1}(\theta_{j+1}) \cdot \Lambda_{\theta_{j+1},b}^{n-j-1,1}(\theta_{j+2})...\Lambda_{\theta_{n-1},b}^{1,1}(\theta_{n})]
    \end{align}

    and

    \begin{align}
        \Lambda_{a,b}^n&(\theta_1,...,\theta_n)\nonumber\\ 
        &= [\Lambda_{a,\theta_2}^{1,1}(\theta_1) \cdot \Lambda_{a,\theta_3}^{2,2}(\theta_2)...\Lambda_{a,\theta_j}^{j-1,j-1}(\theta_{j-1})] \cdot [\Lambda_{a,\theta_{j+1}}^{j,j}(\theta_{j}) \cdot \Lambda_{a,\theta_{j+2}}^{j+1,j+1}(\theta_{j+1})...\Lambda_{a,b}^{n,n}(\theta_{n})].
    \end{align}

    \noindent
    Notice both expressions factor into a part which does not depend on $\theta_1,...,\theta_{j-1}$ and a part which does not depend on $\theta_{j+1},...,\theta_{n}$. As a consequence 


    \begin{equation}
        \frac{\Lambda_{a,b}^{n,1}(\theta_1) \cdot \Lambda_{\theta_1,b}^{n-1,1}(\theta_2)...\Lambda_{\theta_{j-1},b}^{n-j+1,1}(\theta_j)}{\Lambda_{a,\theta_2}^{1,1}(\theta_1) \cdot \Lambda_{a,\theta_3}^{2,2}(\theta_2)...\Lambda_{a,\theta_j}^{j-1,j-1}(\theta_{j-1})} = C(a,\theta_j,b) = \frac{\Lambda_{a,\theta_{j+1}}^{j,j}(\theta_{j}) \cdot \Lambda_{a,\theta_{j+2}}^{j+1,j+1}(\theta_{j+1})...\Lambda_{a,b}^{n,n}(\theta_{n})}{\Lambda_{\theta_j,b}^{n-j,1}(\theta_{j+1}) \cdot \Lambda_{\theta_{j+1},b}^{n-j-1,1}(\theta_{j+2})...\Lambda_{\theta_{n-1},b}^{1,1}(\theta_{n})}
    \end{equation}

    \noindent
    for some function $C$. We tacitly used that the denominators above are non-zero.  However if one of the factors in the denominator of the equation above is zero, then  the two-sided Markov property holds automatically because both left and right hand side of the defining property \eqref{eq:two-sided_Markov_def} are equal to zero. One can now relate the factors of the first and the second expression via $C$. We use this to show the two-sided Markov property.

     \begin{align}
     \label{eq:prop1.3_3to1}
        \Lambda_{a,b}^n&(\theta_1,...,\theta_n)\nonumber\\ 
        &= [\Lambda_{a,b}^{n,1}(\theta_1) \cdot \Lambda_{\theta_1,b}^{n-1,1}(\theta_2)...\Lambda_{\theta_{j-1},b}^{n-j+1,1}(\theta_j)] \cdot [\Lambda_{\theta_j,b}^{n-j,1}(\theta_{j+1}) \cdot \Lambda_{\theta_{j+1},b}^{n-j-1,1}(\theta_{j+2})...\Lambda_{\theta_{n-1},b}^{1,1}(\theta_{n})]\\
        &= C(a,\theta_j,b) \cdot [\Lambda_{a,\theta_2}^{1,1}(\theta_1) \cdot \Lambda_{a,\theta_3}^{2,2}(\theta_2)...\Lambda_{a,\theta_{j}}^{j-1,j-1}(\theta_{j-1})] \cdot [\Lambda_{\theta_j,b}^{n-j,1}(\theta_{j+1}) \cdot \Lambda_{\theta_{j+1},b}^{n-j-1,1}(\theta_{j+2})...\Lambda_{\theta_{n-1},b}^{1,1}(\theta_{n})]\nonumber\\
        &=C(a,\theta_j, b) \cdot \Lambda_{a,\theta_j}^{j-1}(\theta_1,...,\theta_{j-1}) \cdot \Lambda_{\theta_j, b}^{n-j}(\theta_{j+1},...,\theta_{n}) \nonumber.
    \end{align}

\noindent
Integrating both sides of this equation with respect to $\theta_1,...,\theta_{j-1},\theta_{j+1},...,\theta_n$ yields $C(a,\theta_j, b) =\Lambda^{n,j}_{a,b}(\theta_j)$, which then gives us the two-sided Markov property in \eqref{eq:prop1.3_3to1}. 
    
\end{proof}

\begin{remark}
\label{rmk:construction}
    Proposition \eqref{prop:characterization_prod} will play an important role in later applications since it shows the whole family is determined by its left and right-side marginals, i.e. $\{ \Lambda^{n,1}\}_{n \geq 1}$ and $\{ \Lambda^{n,n}\}_{n \geq 1}$. More precisely, suppose one has a family up to some $m \geq 0$, $\{\Lambda^n\}_{m \geq n \geq 1}$. Then we can construct a valid density $\Lambda^{m+1}$ if we can find $\Lambda^{m+1,1}$ and $\Lambda^{m+1,m+1}$ such that

    \begin{equation}
        \Lambda^{m+1}_{a,b}(\theta_1,...,\theta_{m+1}) :=\Lambda_{a,b}^{m+1,1}(\theta_1) \cdot \Lambda_{\theta_1,b}^{m}(\theta_2,...,\theta_{m+1}) =   \Lambda_{a,\theta_{m+1}}^{m}(\theta_1,...,\theta_m) \cdot \Lambda_{a,b}^{m+1,m+1}(\theta_{m+1}).
    \end{equation}
    
    \noindent
    for all $a,b \in \mathbb{R}^+$ and $\theta_1,...,\theta_{m+1} \in I_{a,b}$. Indeed, one can substitute the product forms in \eqref{eq:prod_forms} for $\Lambda_{\theta_1,b}^{m}(\theta_2,...,\theta_{m+1})$ and $\Lambda_{a,\theta_{m+1}}^{m}(\theta_1,...,\theta_m)$. Then $\Lambda_{a,b}^{m+1}$ again has the product form

    \begin{equation}
        \Lambda^{m+1}_{a,b}(\theta_1,...,\theta_{m+1}) = \prod_{i=0}^{m} \Lambda_{\theta_i,b}^{m+1-i,1}(\theta_{i+1}) = \prod_{i=0}^{m} \Lambda_{a,\theta_{m+2-i}}^{m+1-i,m+1-i}(\theta_{m+1-i}),
    \end{equation}

    \noindent
    and hence satisfies the two-side Markov property.
    
\end{remark}

\subsection{Implications of the symmetry property}
In this subsection we examine some of the consequences of the symmetry property defined in \eqref{eq:symm_def}. First, we will relate the $j$-th and the $n-j+1$-th marginal. This is done in proposition \ref{prop:relation_marginals}. Next we show that the symmetry of a family $\{\Lambda^n\}_{n\geq 1}$ is completely determined by the symmetry of $\Lambda^1$, provided that the family is two-sided Markov. This result is the content of proposition \ref{prop:propagation_sym}.

\begin{proposition}
\label{prop:relation_marginals}
    Let $\{\Lambda^n\}_{n\geq1}$be symmetric. Assume $a,b \in \mathbb{R}^+$, $\theta \in \mathbb{R}$ and $n \geq 1$, then 

    \begin{equation}
        \Lambda_{a,b}^{n,j}(\theta) = \Lambda_{b,a}^{n,n-j+1}(\theta)
    \end{equation}

    for all $j\in\{1,...,n\}$.
\end{proposition}

\begin{proof}
    \begin{align}
        \Lambda_{a,b}^{n,j}(\theta) &= \int_{-\infty}^{\infty}...\int_{-\infty}^{\infty} \Lambda_{a,b}^{n}(\theta_1,...,\theta, ..., \theta_n) d\theta_1...d\theta_{j-1} d\theta_{j+1} ... d\theta_n\\
        &=\int_{-\infty}^{\infty}...\int_{-\infty}^{\infty} \Lambda_{b,a}^{n}(\theta_n,...,\theta, ..., \theta_1) d\theta_1...d\theta_{j-1} d\theta_{j+1} ... d\theta_n =  \Lambda_{b,a}^{n,n-j+1}(\theta)\nonumber
    \end{align}
\end{proof}

\begin{proposition}
\label{prop:propagation_sym}
    Let $\{\Lambda^n_{a,b}\}_{n \geq 1}$ be two-sided Markov with $\Lambda_{a,b}^1 = \Lambda_{b,a}^1$, then the whole family $\{\Lambda_{a,b}^n\}_{n \geq 1}$ is symmetric.
\end{proposition}

\begin{proof}
    We use induction. The symmetry of $\Lambda_{a,b}^n$ holds for $n=1$ by assumption. Suppose the symmetry holds for $n=k$ and that $(\theta_1,..., \theta_{k+1}) \in \text{supp } \Lambda_{a,b}^{k+1}$. Notice that, due to the two-sided Markov property, 

    \begin{equation}
        1 = \frac{\Lambda_{a,b}^{k+1}(\theta_1,...,\theta_{k+1})}{\Lambda_{a,b}^{k+1}(\theta_1,...,\theta_{k+1})} = \frac{\Lambda_{a,b}^{k+1,1}(\theta_1) \cdot \Lambda_{\theta_1,\theta_k}^{k-1}(\theta_2,...,\theta_{k}) \cdot \Lambda_{\theta_1,b}^{k,k}(\theta_{k+1})}{\Lambda_{a,b}^{k+1,k+1}(\theta_{k+1}) \cdot \Lambda_{\theta_1,\theta_k}^{k-1}(\theta_2,...,\theta_{k}) \cdot \Lambda_{a,\theta_{k+1}}^{k,1}(\theta_1)}.
    \end{equation}

We obtain

    \begin{equation}
    \label{eq:seperability}
        \frac{\Lambda_{a,b}^{k+1,1}(\theta_1)}{\Lambda_{a,b}^{k+1,k+1}(\theta_{k+1})} = \frac{\Lambda_{a,\theta_{k+1}}^{k,1}(\theta_1)}{\Lambda_{\theta_1,b}^{k,k}(\theta_{k+1})}.
    \end{equation}

    \noindent
    We will use this equation to show the symmetry for $n= k+1$. Notice that we have symmetry if 

    \begin{equation} \label{eq:criterium_symm}
        1 = \frac{\Lambda_{a,b}^{k+1}(\theta_1,...,\theta_{k+1})}{\Lambda_{b,a}^{k+1}(\theta_{k+1},...,\theta_1)}= \frac{\Lambda_{a,b}^{k+1,1}(\theta_1) \cdot \Lambda_{\theta_1,b}^{k}(\theta_2,...,\theta_{k+1})}{\Lambda_{b,a}^{k+1,k+1}(\theta_1) \cdot \Lambda_{b,\theta_1}^{k}(\theta_{k+1},...,\theta_{2})}
    \end{equation}

    By the induction hypothesis we have $\Lambda_{\theta_1,b}^{k}(\theta_2,...,\theta_{k+1}) = \Lambda_{b,\theta_1}^{k}(\theta_{k+1},...,\theta_{2})$, hence the right hand side of \eqref{eq:criterium_symm} simplifies. We see that symmetry for $n=k+1$ is equivalent with

    \begin{equation}
    \label{eq:symmetry_condition}
        \frac{\Lambda_{a,b}^{k+1,1}(\theta_1)}{\Lambda_{b,a}^{k+1,k+1}(\theta_1)} = 1.
    \end{equation}

    \noindent
    We will prove \eqref{eq:symmetry_condition}. We start from \eqref{eq:seperability} and observe that

    \begin{equation}
        \frac{\Lambda_{a,b}^{k+1,1}(\theta_1)}{\Lambda_{a,b}^{k+1,k+1}(\theta_{k+1})} = \frac{\Lambda_{a,\theta_{k+1}}^{k,1}(\theta_1)}{\Lambda_{\theta_1,b}^{k,k}(\theta_{k+1})}  =  \frac{ \Lambda_{\theta_{k+1},a}^{k,k}(\theta_1)}{\Lambda_{b,\theta_1}^{k,1}(\theta_{k+1})} = \frac{\Lambda_{b,a}^{k+1,k+1}(\theta_1)}{\Lambda_{b,a}^{k+1,1}(\theta_{k+1})}.
    \end{equation}

    \noindent
    For the second equality we use the induction hypothesis combined with Proposition \ref{prop:relation_marginals} and for the last equality we again use \eqref{eq:seperability}. One can now separate this equation into a $\theta_1$-dependent and a $\theta_{k+1}$-dependent side. That is

    \begin{equation}
        \frac{\Lambda_{a,b}^{k+1,1}(\theta_1)}{\Lambda_{b,a}^{k+1,k+1}(\theta_1)} = \frac{\Lambda_{a,b}^{k+1,k+1}(\theta_{k+1})}{\Lambda_{b,a}^{k+1,1}(\theta_{k+1})} = c
    \end{equation}

    \noindent
    for some $c$ independent from both $\theta_1$ and $\theta_{k+1}$. Integrating $\Lambda_{a,b}^{k+1,1}(\theta_1) = c \cdot \Lambda_{b,a}^{k+1,k+1}(\theta_1)$ over $\theta_1$ yields $c = 1$. This yields \eqref{eq:symmetry_condition} and hence concludes the proof.
\end{proof}

\section{Support of the densities}
In this section we study the support of the densities $\{\Lambda^n_{a,b}\}_{n \geq 1}$. More precisely, we show that the support of the first marginal $\Lambda^{n,1}_{a,b}$ is given by $I_{a,b}$ for all $n\in \mathbb{N}$. The product form from the previous section then implies that $\Lambda^n_{a,b}$ is supported on the set of ordered $n$-tuples in the interval $I_{a,b}$.\\

\noindent
For two-sided Markov families $\{\Lambda^n\}_{n \geq 1}$ one can find transition operators acting on the marginals. More precise, for suitable $i$, one can find a raising operator $P_{n,i+1}$ and a lowering operator $Q_{n,i-1}$ transforming $\Lambda_{a,b}^{n,i}$ to $\Lambda_{a,b}^{n,i+1}$ and $\Lambda_{a,b}^{n,i-1}$ respectively. The proposition below states these transition operators explicitly.

\begin{proposition}
\label{prop:transition_op_marginals}
    Let $\{\Lambda^n\}_{n\geq1}$ be two-sided Markov. For $a,b \in \mathbb{R}^+$ and $n \geq 1$, $i\in \{1,...,n\}$ we have

    \begin{equation}
    \begin{cases}
         \Lambda_{a,b}^{n,i+1} = P_{n,i+1} \Lambda_{a,b}^{n,i} \quad \text{ for } i\in \{1,...,n-1\}\\
         \Lambda_{a,b}^{n,i-1} = Q_{n,i-1}\Lambda_{a,b}^{n,i} \quad \text{ for } i\in \{2,...,n\}
    \end{cases}.
    \end{equation}

    Here $P_{n,i+1}: L^1(-\infty, \infty) \rightarrow L^1(-\infty, \infty)$ and $Q_{n,i-1}: L^1(-\infty, \infty) \rightarrow L^1(-\infty, \infty)$ are given by

    \begin{equation}
    \begin{cases}
        P_{n,i+1} f (x) = \int_{-\infty}^\infty f(y) \cdot \Lambda_{y,b}^{n-i,1}(x) dy\\
        Q_{n,i-1} f (x) = \int_{-\infty}^\infty f(y) \cdot \Lambda_{a,y}^{i-1,i-1}(x) dy 
    \end{cases}.
    \end{equation}

    \begin{proof}
        It is straightforward to show that transition operators $P_{n,i+1}$ and $Q_{n,i-1}$ map $L^1(-\infty, \infty)$ functions to $L^1(-\infty, \infty)$ functions. Indeed, let $f \in L^1(-\infty, \infty)$, then

        \begin{equation}
            \int_{-\infty}^\infty |P_{n,i+1} f(x)| dx \leq \int_{-\infty}^\infty \int_{-\infty}^\infty |f(y)| \cdot |\Lambda_{y,b}^{n-i,1}(x)| dy dx \leq ||f||_1< \infty.\
        \end{equation}

        A similar argument can be used to show $Q_{n,i-1}$ map $L^1(-\infty, \infty)$ functions to $L^1(-\infty, \infty)$.\\

        \noindent
        We now show that $\Lambda_{a,b}^{n,i+1} = P_{n,i+1} \Lambda_{a,b}^{n,i}$. On can use an analogous proof to show $ \Lambda_{a,b}^{n,i-1} = Q_{n,i-1}\Lambda_{a,b}^{n,i}$. \\

        \noindent
        We use the notation $\Lambda_{a,b}^{n,(i,i+1)}$ to denote the joint distribution of $\Theta_i$ and $\Theta_{i+1}$, i.e.

        \begin{equation}
            \Lambda_{a,b}^{n,(i,i+1)}(\theta_{i},\theta_{i+1}) = \int_{-\infty}^{\infty} ... \int_{-\infty}^{\infty}  \Lambda_{a,b}^{n}(\theta_1,...,\theta_n) d\theta_1,...,d\theta_{i-1} d\theta_{i+2}...d\theta_n.
        \end{equation}

        \noindent
        We write $ \Lambda_{a,b}^{n,(i,i+1)}(y,\theta_{i+1}|\Theta_{i} = y)$ for the joint distribution under the condition $\Theta_{i}=y$. Suppose we have 
        \begin{align}
        \label{eq:intermediate_step_trans_op}
             \Lambda_{a,b}^{n,(i,i+1)}(y,\theta_{i+1}|\Theta_{i} = y) = \Lambda_{y,b}^{n-i,1}(\theta_{i+1}),
        \end{align}

        then the result follows directly:

        \begin{align}
            \Lambda_{a,b}^{n,i+1}(\theta_{i+1}) &= \int_{-\infty}^{\infty}\Lambda_{a,b}^{n,i}(y)  \cdot \Lambda_{a,b}^{n,(i,i+1)}(y,\theta_{i+1}|\Theta_{i} = y) dy\\
            &= \int_{-\infty}^{\infty}\Lambda_{a,b}^{n,i}(y)  \cdot \Lambda_{y,b}^{n-i,1}(\theta_{i+1}) dy\nonumber\\
            &= P_{n,i+1}  \Lambda_{a,b}^{n,i}(\theta_{i+1}) \nonumber.
        \end{align}

        To prove \eqref{eq:intermediate_step_trans_op} we compute $\Lambda_{a,b}^{n,(i,i+1)}(y,\theta_{i+1}|\Theta_{i} = y)$,

        \begin{align}
            \Lambda_{a,b}^{n,(i,i+1)}(y,\theta_{i+1}|&\Theta_{i} = y) \nonumber\\ 
            =& \int_{-\infty}^\infty ...\int_{-\infty}^\infty \Lambda_{a,b}^n(\theta_1,..., \theta_{i-1}, y, \theta_{i+1},...,\theta_n |\Theta_{i} = y) d\theta_1...d\theta_{i-1}d\theta_{i+2}...d\theta_n\\
            =& \int_{-\infty}^\infty ...\int_{-\infty}^\infty \Lambda_{a,y}^{i-1}(\theta_1,..., \theta_{i-1}) \cdot \Lambda_{y,b}^{n-i}(\theta_{i+1},..., \theta_n) d\theta_1...d\theta_{i-1}d\theta_{i+2}...d\theta_n\nonumber\\
            =& \int_{-\infty}^\infty ...\int_{-\infty}^\infty \Lambda_{y,b}^{n-i}(\theta_{i+1},..., \theta_n) d\theta_{i+2}...d\theta_n\nonumber\\
            &=\Lambda_{y,b}^{n-i,1}(\theta_{i+1}).\nonumber
        \end{align}
    \end{proof}
\end{proposition}

\noindent
In the following proposition we make use of the transition operators introduced in \ref{prop:transition_op_marginals} in order to prove that the support of $\Lambda^1_{a,b}$ determines the support of the marginals $\Lambda^{n,1}_{a,b}$ and hence via the product structure the support $\Lambda^n_{a,b}$.


\begin{proposition}
\label{prop:supp_lambdas}
    Let $\{\Lambda^n\}_{n\geq1}$ be two-sided Markov and $a,b \in \mathbb{R}^+$. Then
    
    \begin{equation}
        \text{\normalfont supp }(\Lambda_{a,b}^1) = I_{a,b} \quad \text{implies} \quad \text{\normalfont supp }(\Lambda_{a,b}^{n,1}) = I_{a,b}
    \end{equation}

    for all $n \geq 1$.
\end{proposition}

\begin{proof}
    We use induction. For $n=1$ the statement holds by assumption. Assume $\forall m \leq k$: supp $\Lambda_{a,b}^{m,1} = I_{a,b}$. First we show that supp $\Lambda_{a,b}^{k+1,1} \subset I_{a,b}$. This part of the argument doesn't rely on the transition operators, those will be used in step 2.\\

    \noindent
    \textbf{Step 1: the support of $\Lambda_{a,b}^{k+1,1}$ is a subset of $I_{a,b}$}\\
    
    \noindent
    Suppose $\theta_1 \in \mathbb{R}$ is in the support of $\Lambda_{a,b}^{k+1,1}$, i.e. $\Lambda_{a,b}^{k+1,1}(\theta_1) \neq 0$. Then,there exist $\theta_2,...,\theta_{k+1}$ such that

    \begin{align}
    \label{eq:two_product_expansions}
        \Lambda_{a,b}^{k+1}(\theta_1,\theta_2,...,\theta_{k+1}) &= \Lambda_{a,b}^{k+1,1}(\theta_1) \cdot \Lambda_{\theta_1,b}^{k,1}(\theta_2) ... \Lambda_{\theta_k,b}^{1,1}(\theta_{k+1})\\
        &= \Lambda_{a,b}^{k+1,k+1}(\theta_{k+1}) \cdot \Lambda_{a,\theta_{k+1}}^{k,k}(\theta_k) ... \Lambda_{a,\theta_2}^{1,1}(\theta_1) \neq 0 \nonumber.
    \end{align}
    
    \noindent
    Suppose $\theta_1 < b$, then $\theta_1 < \theta_2 < ... < \theta_{k+1} \leq b$ by the induction hypothesis and the first line of \eqref{eq:two_product_expansions}. The second line of \eqref{eq:two_product_expansions} tells us that either $a < \theta_{1} < \theta_2 < ... < \theta_{k+1}$ or  $a > \theta_{1} > \theta_2 > ... > \theta_{k+1}$. Combining the information from line one and line two gives $a < \theta_1 < b$. Similarly we can assume $\theta_1 < b$ and derive $b < \theta_1 < a$. We conclude $\theta_1 \in I_{a,b}$.\\

    \noindent
    \textbf{Step 2: $I_{a,b}$ is the support of $\Lambda_{a,b}^{k+1,1}$}\\
    We notice that $\Lambda_{a,b}^{k+1,1}$ is a fixed point for $Q_{k+1,1} P_{k+1,2}$ by virtue of proposition \ref{prop:transition_op_marginals},

    \begin{equation}
        Q_{k+1,1} \Big(P_{k+1,2} \Lambda_{a,b}^{k+1,1}\Big) = Q_{k+1,1} \Lambda_{a,b}^{k+1,2} = \Lambda_{a,b}^{k+1,1}.
    \end{equation}

    \noindent
    We simplify notation by writing $f$ for $\Lambda_{a,b}^{k+1,1}$, $g$ for $P_{k+1,2} \Lambda_{a,b}^{k+1,1}$, i.e. $f:= \Lambda_{a,b}^{k+1,1}$ and $g := P_{k+1,2} \Lambda_{a,b}^{k+1,1} = \Lambda_{a,b}^{k+1,2}$.\\
    

    \noindent
    Assume, without loss of generality, that $a \leq b$. Then,

    \begin{equation}
        g(x) = \int_{-\infty}^{\infty} f(y) \cdot \Lambda_{y,b}^{k,1} (x) dy = \int_a^b f(y) \cdot \Lambda_{y,b}^{k,1} (x) dy.
    \end{equation}

    \noindent
    Here we use that $\text{supp }f= [a,b]$. Note that under the induction hypothesis supp $\Lambda_{y,b}^{k,1} = [y,b]$ for $y \in [a,b]$ and that supp $f \subset I_{a,b} = [a,b]$. We see that 

    \begin{equation}
    \label{eq:supp_g}
        \text{supp } g = \{x \in \mathbb{R}: \text{ess inf} (\text{supp } f) < x < b \},
    \end{equation}

    \noindent
    where $\text{ess inf} (\text{supp } f) := \inf \big \{ z \in \mathbb{R}: \lambda\big( (-\infty,z] \cap \text{supp} f \big) \neq 0 \big\}$ with $\lambda$ the Lebesgue measure. Analogously, we have

    \begin{equation}
        f(x) = Q_{k+1,1} g (x) = \int_{-\infty}^{\infty} g(y) \cdot \Lambda_{a,y}^{1} (x) dy = \int_a^b g(y) \cdot \Lambda_{a,y}^{1} (x) dy
    \end{equation}

    and

    \begin{equation}
        \text{supp } (f) = \{ x \in \mathbb{R}: a < x < \text{ess sup} (\text{supp } g) \}
    \end{equation}

    \noindent
    with $\text{ess sup} (\text{supp } g):= \sup \big \{ z \in \mathbb{R}:\lambda\big( [z,\infty) \cap \text{supp } g \big) \neq 0 \big\}$. Notice that $\text{supp } g$ can not be empty since $f$ is a fixed point of $Q_{k+1,1} P_{k+1,2}$. Indeed, if $\text{supp } g$ would be empty then $f$ would be zero, which is not possible since $f = \Lambda_{a,b}^{k+1,1}$ is a density function. This implies ess sup(supp $g$) $= b$ since the support of $g$ is the interval in \eqref{eq:supp_g}. We conclude supp $f$ = supp $\Lambda_{a,b}^{k+1,1} = (a,b) = I_{a,b}$.
    
\end{proof}

\begin{remark}
    The support of the first marginals determines directly the support of the joint density. Indeed, the product expressions for $\Lambda_{a,b}^{n}$ in proposition \ref{prop:characterization_prod} directly yields that 

    \begin{equation}
        \text{\normalfont supp } \Lambda_{a,b}^{n} = 
        \begin{cases}
            \{\theta_1,...,\theta_{n} \in \mathbb{R}: a < \theta_1 < ... < \theta_n < b\} \quad \text{for } a<b\\
            \{\theta_1,...,\theta_{n} \in \mathbb{R}: a > \theta_1 > ... > \theta_n > b\}\quad \text{for } b < a
        \end{cases}.
    \end{equation}
\end{remark}

\section{Uniqueness and Recursive construction}
In this subsection we attempt to construct families $\{\Lambda^n\}_{n \geq 1}$ satisfying the two-sided Markov property. In particular, we are interested in identifying which ``one-site systems'' $\Lambda^1$ extend to a full two-sided Markov family $\{\Lambda^n\}_{n \geq 1}$. We call those \emph{extendable}. So far, in section 3,  we have seen that both symmetry and the support of the first marginal
$\Lambda^{n,1}$ are determined by $\Lambda^1$. One of the results in this section, Theorem \ref{thm:uniqueness}, gives uniqueness for extensions of $\Lambda^1$ which have support on $I_{a,b}$ for $a,b \in \mathbb{R}$. Moreover, we will use proposition \ref{prop:characterization_prod} to identify the form of the $\Lambda^1$ which are extendable.

\subsection{Uniqueness}
In the previous subsection we found transition operators with the marginals as fixed points. Since these operators are defined in terms of marginals of smaller order, one could be tempted to show uniqueness recursively via a fixed point theorem (e.g the Krein-Rutman theorem). However, in our setting it appears to be non-trivial to satisfy the condition for such theorems. Therefore we give a more direct approach below, relying strongly on our knowledge about the support of the marginals.

\begin{theorem} \label{thm:uniqueness}
    Let $\{\Lambda^n\}_{n \geq 1}$ be a two-sided Markov family. Then $\Lambda^1$ uniquely determines $\{\Lambda^n\}_{n \geq 1}$.
\end{theorem}

\begin{proof}
Suppose we have two families $\{\Lambda^n\}_{n \geq 1}$ and $\{\widetilde{\Lambda}^n\}_{n \geq 1}$ which are both two-sided Markov. Moreover assume 

    \begin{equation}
        \Lambda^1_{a,b} = \widetilde{\Lambda}^1_{a,b}.
    \end{equation}

    \noindent
    We show

     \begin{equation}
         \{\Lambda^n\}_{n \geq 1} = \{\widetilde{\Lambda}^n\}_{n \geq 1}.
     \end{equation}

    \noindent
    using induction. For $n = 1$, the density functions are equal by assumption. Let $a,b \in \mathbb{R}^+, a < b$, be arbitrary and assume that $\Lambda^k = \widetilde{\Lambda}^k$. For any $a < \theta_1 < ... < \theta_{k+1} < b$ we have

    \begin{align}
        0 \neq \Lambda^{k+1}_{a,b}(\theta_1, ...,\theta_{k+1}) &= \Lambda^{k+1,1}_{a,b}(\theta_1) \cdot \Lambda_{\theta_1, \theta_k}^{k-1}(\theta_2,...,\theta_k) \cdot \Lambda_{\theta_1,b}^{k,k}(\theta_{k+1})\\
        &=\Lambda^{k+1,k+1}_{a,b}(\theta_{k+1}) \cdot \Lambda_{\theta_1, \theta_k}^{k-1}(\theta_2,...,\theta_k) \cdot \Lambda_{a,\theta_{k+1}}^{k,1}(\theta_{1})\nonumber
    \end{align}

    and

    \begin{align}
        0 \neq \widetilde{\Lambda}^{k+1}_{a,b}(\theta_1, ...,\theta_{k+1}) &= \widetilde{\Lambda}^{k+1,1}_{a,b}(\theta_1) \cdot \Lambda_{\theta_1, \theta_k}^{k-1}(\theta_2,...,\theta_k) \cdot \Lambda_{\theta_1,b}^{k,k}(\theta_{k+1})\\
        &=\widetilde{\Lambda}^{k+1,k+1}_{a,b}(\theta_{k+1}) \cdot \Lambda_{\theta_1, \theta_k}^{k-1}(\theta_2,...,\theta_k) \cdot \Lambda_{a,\theta_{k+1}}^{k,1}(\theta_{1})\nonumber.
    \end{align}

    We find

    \begin{equation}
    \label{eq:devision_marginals}
        \frac{\Lambda^{k+1,1}_{a,b}(\theta_1)}{\Lambda^{k+1,k+1}_{a,b}(\theta_{k+1})} = \frac{\Lambda_{a,\theta_{k+1}}^{k,1}(\theta_{1})}{\Lambda_{\theta_1,b}^{k,k}(\theta_{k+1})} =  \frac{\widetilde{\Lambda}^{k+1,1}_{a,b}(\theta_1)}{\widetilde{\Lambda}^{k+1,k+1}_{a,b}(\theta_{k+1})}.
    \end{equation}

    Let $\epsilon >0$ be arbitrary and pick $\theta_{k+1} = b-\epsilon$. For all $\theta_1 \in (a,b-\epsilon)$  

    \begin{equation}
    \label{eq:comparison}
        \frac{\Lambda^{k+1,1}_{a,b}(\theta_1)}{\Lambda^{k+1,k+1}_{a,b}(b-\epsilon)} =  \frac{\widetilde{\Lambda}^{k+1,1}_{a,b}(\theta_1)}{\widetilde{\Lambda}^{k+1,k+1}_{a,b}(b-\epsilon)}.
    \end{equation}

    \noindent
    We see that $\Lambda^{k+1,1}_{a,b}$ and $\tilde{\Lambda}^{k+1,1}_{a,b}$ are the same up to a constant factor on the interval $(a,b- \epsilon)$. Since this holds for arbitrary $\epsilon$ and $\Lambda^{k+1,1}_{a,b}$, $\tilde{\Lambda}^{k+1,1}_{a,b}$ are normalized, we can conclude that $\Lambda^{k+1,1}_{a,b} = \widetilde{\Lambda}^{k+1,1}_{a,b}$ on all of $(a,b)$. Moreover, we have $\Lambda_{a,b}^{k+1} = \widetilde{\Lambda}_{a,b}^{k+1}$. For $b \leq a$ a similar argument holds.
\end{proof}


\begin{remark}
    The theorem above essentially says that it is impossible for different two-sided Markov families to have the same single site density. However, which $\Lambda^1$ correspond to a two-sided Markov family, i.e., which $\Lambda^1$ are extendable  is at this point an open question. This will be the subject of the next subsection, in particular in Theorem \ref{thm:characterization_two-sided_symmetric} below we give a full characterization. 
\end{remark}

\subsection{Recursive construction}
By virtue of proposition \ref{prop:characterization_prod} we can attempt to recursively construct a family $\{\Lambda^n\}_{n\geq 1}$ from a suitable, symmetric, $\Lambda^1$. In the proof of the theorem below we find a construction as described in remark \ref{rmk:construction}, i.e. we find suitable left and right marginals to construct $\Lambda^{n+1}$ from $\Lambda^n$. before we state this theorem, we first prove two preliminary results.

\begin{lemma}
\label{lemma:separability}
    Let $\{\Lambda^n\}_{n \geq 1}$ be a symmetric two-sided Markov family. For arbitrary $n\geq 1$ and $y\in \mathbb{R}^+$, there exist $y$-dependent functions $f_n,g_n$ and $h_n$ on $(\mathbb{R}^+)^2$ such that 

    \begin{equation}
        \Lambda_{a,b}^{n,1}(x) = f_n(a,b) \cdot g_n(b,x) \cdot h_n(x,a) \cdot \mathds{1}(x \in I_{a,b})
    \end{equation}

    for all $a,x,b \in \mathbb{R}^+$ such that $x \in I_{a,b}$ and $b \in I_{a,y}$. If a density is of this form we will refer to it as factorizable.
\end{lemma}

\begin{proof}
By conditioning on $\theta_1$ and $\theta_{n+1}$ we find

\begin{align}
\label{eq:ping}
    \Lambda_{a,y}^{n+1,1}(\theta_1) \cdot \Lambda_{\theta_1,\theta_{n+1}}^{n-1}(\theta_2,...,\theta_n) \cdot  \Lambda_{\theta_1,y}^{n,n}(\theta_{n+1}) = \Lambda_{a,\theta_{n+1}}^{n,1}(\theta_1) \cdot \Lambda_{\theta_1,\theta_{n+1}}^{n-1}(\theta_2,...,\theta_{n}) \cdot  \Lambda_{a,y}^{n+1,n+1}(\theta_{n+1}), 
\end{align}

\noindent
where both sides of the equation are equal to $\Lambda^{n+1}_{a,y}(\theta_1,...,\theta_{n+1})$.
\noindent
To proceed we need information on the support of $\Lambda_{a,y}^{n+1}$. To this end we assume $a < y$. As a consequence, we can pick $a < \theta_1 < ... < \theta_{n+1} < y$, then 

\begin{equation}
\label{eq:separability_symm}
    \frac{\Lambda_{a,y}^{n+1,1}(\theta_1)}{\Lambda_{y,a}^{n+1,1}(\theta_{n+1})} = \frac{\Lambda_{a,y}^{n+1,1}(\theta_1)}{\Lambda_{a,y}^{n+1,n+1}(\theta_{n+1})} = \frac{\Lambda_{a,\theta_{n+1}}^{n,1}(\theta_1)}{\Lambda_{\theta_1,y}^{n,n}(\theta_{n+1})} = \frac{\Lambda_{a,\theta_{n+1}}^{n,1}(\theta_1)}{\Lambda_{y,\theta_1}^{n,1}(\theta_{n+1})}.
\end{equation}

\noindent
The first and third equality use the symmetry property. The second equality follows from \eqref{eq:ping}. The same equation can be derived for $y<a$ and $\theta_1 > \theta_{n+1}$, hence we continue with general $a,y \in \mathbb{R}^+$ and $\theta_1 \in I_{a,\theta_{n+1}}$, $\theta_{n+1} \in I_{a,y}$. Notice that in the case $a<y$ we have $a<x<b<y$, whereas in the case $y<a$ we have $y<b<x<a$. That is, in the first case $y$ acts as an upper bound for the interval $I_{a,b}$ and in the second case it acts like a lower bound. We observe that 

\begin{equation}
    \Lambda_{a,\theta_{n+1}}^{n,1}(\theta_1) =  [\Lambda_{y,a}^{n+1,1}(\theta_{n+1})]^{-1} \cdot \Lambda^{n,1}_{y,\theta_1}(\theta_{n+1}) \cdot \Lambda_{a,y}^{n+1,1}(\theta_1).
\end{equation}

\noindent
One can now take $\theta_1 = x$ and $\theta_{n+1} = b$ to obtain

\begin{equation}
    \Lambda_{a,b}^{n,1}(x) =  [\Lambda_{y,a}^{n+1,1}(b)]^{-1} \cdot \Lambda^{n,1}_{y,x}(b) \cdot \Lambda_{a,y}^{n+1,1}(x).
\end{equation}

\noindent
We can then define the functions $f_n, g_n$ and $h_n$ as follows

\begin{equation}
    f_n(a,b) := [\Lambda_{y,a}^{n+1,1}(b)]^{-1}, \quad g_n(b,x) := \Lambda^{n,1}_{y,x}(b), \quad h_n(x,a) := \Lambda_{a,y}^{n+1,1}(x).
\end{equation}

\end{proof}

\begin{lemma}
\label{lemma:symmetric_f_g}
    Let $\{\Lambda^n\}_{n \geq 1}$ be a symmetric two-sided Markov family. For arbitrary $z_1,z_2 \in \mathbb{R}^+$, $z_1 \leq z_2$, there exist functions $f,g$ and $h$ on $\mathbb{R}$ such that 

    \begin{equation}
        \Lambda_{a,b}^1(x) = f(a,b) \cdot g(b,x) \cdot h(x,a) \cdot \mathds{1}(x \in I_{a,b})
    \end{equation}

    for all $a,b \in (z_1,z_2)$. Moreover, the functions $f$ and $g$ can be chosen to  be symmetric.
\end{lemma}

\begin{proof}
    Let $Z_1 > 0$ and $Z_2 > 0$ be such that $Z_1 < z_1$ and $z_2< Z_2$. Let $a,b \in (Z_1,Z_2)$ and assume $a \leq b$. Then
    $b \in I_{a,Z_2}$, hence proposition \ref{lemma:separability} gives us functions $\widehat{f}, \widehat{g}$ and $\widehat{h}$ such that 

    \begin{equation}
        \Lambda_{a,b}^1(x) = \widehat{f}(a,b) \cdot \widehat{g}(b,x) \cdot \widehat{h}(x,a) \cdot \mathds{1}(x \in I_{a,b}).
    \end{equation}

    \noindent
    Similarly we can assume $b \leq a$, then $b \in I_{Z_1,a}$, hence there are functions $\widetilde{f}, \widetilde{g}$ and $\widetilde{h}$ such that

    \begin{equation}
        \Lambda_{a,b}^1(x) = \widetilde{f}(a,b) \cdot \widetilde{g}(b,x) \cdot \widetilde{h}(x,a) \cdot \mathds{1}(x \in I_{a,b}).
    \end{equation}
    \newpage

    hence we can define $f,g$ and $h$ as 

    \begin{multicols}{3}
        \begin{equation*}
            f(a,b)=
            \begin{cases}
                \widehat{f}(a,b) \quad \text{if } a < b\\
                \widetilde{f}(a,b) \quad \text{if } b < a
            \end{cases}
        \end{equation*}\break
        \begin{equation*}
            g(b,x)=
            \begin{cases}
                \widehat{g}(b,x) \quad \text{if } x < b\\
                \widetilde{g}(b,x) \quad \text{if } b < x
            \end{cases}
      \end{equation*}\break
      \begin{equation*}
            h(x,a)=
            \begin{cases}
                \widehat{h}(x,a) \quad \text{if } a < x\\
                \widetilde{g}(x,a) \quad \text{if } x < a
            \end{cases}
      \end{equation*}
\end{multicols}

\noindent
This show the first part of the lemma. we proceed with general $a,b \in I_{Z_1,Z_2}$ and $x \in I_{a,b}$
Now we only have to show that we can choose $f$ and $g$ symmetric. We define $\psi(a,x) = \sqrt{g(a,x) \cdot h(x,a)}$ and $\phi(a,b) = \sqrt{f(a,b) \cdot f(b,a)}$. Notice that, using the symmetry of the family $\{\Lambda^n\}_{n\geq1}$,

\begin{align}
\label{eq:construction_phi_psi}
    \Lambda_{a,b}^1(x) &= \sqrt{\Lambda_{a,b}^1(x)\cdot \Lambda_{b,a}^1(x)}\\
    &= \sqrt{f(a,b) \cdot f(b,a)} \cdot \sqrt{g(a,x) \cdot h(x,a)} \cdot  \sqrt{g(b,x) \cdot h(x,b)} \nonumber \\
    &=\phi(a,b) \cdot \psi(a,x) \cdot \psi(b,x). \nonumber
\end{align}

\noindent
We now restrict $a,b$ to 
$(z_1,z_2)$ and claim that $\psi(y,x) / \psi(x,y)$ is factorizable in $x$ and $y$ given the order of $x$ and $y$. More precisely, we claim that

\begin{equation}\label{eq:claim}
    \frac{\psi(y,x)}{\psi(x,y)} = \frac{c_2(y) \mathds{1}(x<y) + c_1(y) \mathds{1}(y<x)}{c_1(x) \mathds{1}(x<y) + c_2(x) \mathds{1}(y<x)}
\end{equation}

\noindent
for some functions $c_1$ and $c_2$ and $x \neq y$, $x,y \in (z_1,z_2)$. Before we prove this claim we demonstrate how we will use it. If the claim holds, then we have

\begin{align}
    \Psi(x,y) :&= [c_1(x) \mathds{1}(x<y) + c_2(x) \mathds{1}(y<x)] \cdot \psi(y,x)\\
    &= [c_2(y) \mathds{1}(x<y) + c_1(y) \mathds{1}(y<x)] \cdot \psi(x,y) = \Psi(y,x) \nonumber
\end{align}

\noindent
for all $x,y \in (z_1,z_2)$ with $x \neq y$. As a consequence

\begin{align} \label{eq:lambda_with_symm_Psi}
    \Lambda_{a,b}^1(x) &= \phi(a,b) \cdot \psi(a,x) \cdot \psi(b,x) \cdot \mathds{1}(x \in I_{a,b}) \\
    &= \phi(a,b) \cdot \Psi(b,x) \cdot  [c_1(x) \cdot \mathds{1}(x<b) + c_2(x) \cdot \mathds{1}(b<x)]^{-1} \psi(a,x) \cdot \mathds{1}(x \in I_{a,b})\nonumber\\
    &= \phi(a,b) \cdot \Psi(b,x) \cdot  \gamma(a,x) \cdot \mathds{1}(x \in I_{a,b})\nonumber, 
\end{align}

where we define $\gamma$, for $x \in I_{a,b}$, as

\begin{align} \label{eq:def_gamma}
    \gamma(a,x) :&= [c_1(x) \cdot \mathds{1}(a<x) + c_2(x) \cdot \mathds{1}(x<a)]^{-1} \psi(a,x)\\
    &= [c_1(x) \cdot \mathds{1}(x<b) + c_2(x) \cdot \mathds{1}(b<x)]^{-1} \psi(a,x). \nonumber
\end{align}

\noindent
Notice that the second equality in \eqref{eq:def_gamma} relies on the fact that $x \in I_{a,b}$. Indeed, for $x \in I_{a,b}$ we have $a<x \iff x<b$ and likewise $b<x \iff x<a$. We can pick $f = \phi$, $g = \Psi$ and $h = \gamma$. Since $\phi$ and $\Psi$ are symmetric, this completes the proof.

Finally, we prove the claim \eqref{eq:claim}.
Recall that \eqref{eq:construction_phi_psi} holds on the whole interval $(Z_1,Z_2)$ rather than only on the subinterval $(z_1,z_2)$. Let $x,y\in I_{z_1,z_2}$. If we have $x < y$ then 

\begin{equation}
    \frac{\Lambda^{2,1}_{z_1,z_2}(x)}{\Lambda^{2,1}_{z_2,z_1}(y)} = \frac{\Lambda^{1}_{z_1,y}(x)}{\Lambda^{1}_{z_2,x}(y)} = \frac{\phi(z_1,y) \cdot \psi(z_1,x) \cdot \psi(y,x)}{\phi(z_2,x) \cdot \psi(z_2,y) \cdot \psi(x,y)}.
\end{equation}

\noindent
Define $C_{1,2}(x) := \frac{\psi(z_1,x)}{\Lambda^{2,1}_{z_1,z_2}(x) \cdot \phi(z_2,x)}$ and $C_{2,1}(y) := \frac{\psi(z_2,y)}{\Lambda^{2,1}_{z_2,z_1}(y) \cdot \phi(z_1,y)}$. The equation above reads then

\begin{equation}
    \frac{C_{2,1}(y)}{ C_{1,2}(x)} = \frac{\psi(y,x)}{\psi(x,y)}.
\end{equation}

Now assume $y < x$, similarly we find

\begin{equation}
    \frac{C_{1,2}(y)}{C_{2,1}(x)} = \frac{\psi(y,x)}{\psi(x,y)}.
\end{equation}

Combining these two expressions yields

\begin{equation}
    \frac{\psi(y,x)}{\psi(x,y)} = \frac{C_{2,1}(y) \cdot \mathds{1}(x<y) + C_{1,2}(y) \cdot \mathds{1}(y<x)}{C_{1,2}(x) \cdot \mathds{1}(x<y) + C_{2,1}(x) \cdot \mathds{1}(y<x)}
\end{equation}

\noindent
We define $c_1 = C_{1,2}$ and $c_2 = C_{2,1}$ to obtain the required form. This proves \eqref{eq:claim}.
\end{proof}
We can now state a characterization of the $\Lambda^1$ which generate a symmetric two-sided Markov family. 
\begin{theorem}
\label{thm:characterization_two-sided_symmetric}


Consider a one-site density $\Lambda^{(1)}$. The following two statements are equivalent:

\begin{enumerate}
    \item There exists a symmetric and two-sided Markov family $\{\Lambda^n\}_{n \geq 0}$ with one-site density $\Lambda^{(1)}$.

    \item $\forall z_1,z_2 \in \mathbb{R}^+$, $a,b \in I_{z_1,z_2}$ there exists a symmetric $g: \mathbb{R}^2 \rightarrow \mathbb{R}^+$ such that 

    \begin{equation}
    \label{eq:form_generating}
        \Lambda^1_{a,b} (x) = f(a,b) \cdot g(b,x) \cdot g(x,a) \cdot \mathds{1}(x \in I_{a,b})
    \end{equation}

    with normalization

    \begin{equation}\label{eq:normalization}
        f(a,b) = \Bigg (\int_{a \wedge b}^{a \vee b}  g(b,y) \cdot g(y,a) dy \Bigg)^{-1}.
    \end{equation} 
\end{enumerate}

\end{theorem}
\begin{proof}
{\bf 1 implies 2.}
Lemma \ref{lemma:symmetric_f_g} states that we can write 

\begin{equation}
    \Lambda^1_{a,b} (x) = \widetilde{f}(a,b) \cdot \widetilde{g}(b,x) \cdot \widetilde{h}(x,a) \cdot \mathds{1}(x \in I_{a,b})
\end{equation}

\noindent
with $\widetilde{f}$ and $\widetilde{g}$ symmetric. Using the symmetry of $\Lambda^1$, we can deduce that $\widetilde{h}$ is of the form $\widetilde{h}(x,a) = c(x) \cdot \widetilde{g}(x,a)$ for some function $c$. Indeed, the symmetry gives

\begin{equation}
    \widetilde{f}(a,b) \cdot \widetilde{g}(b,x) \cdot \widetilde{h}(x,a) = \widetilde{f}(b,a) \cdot \widetilde{g}(a,x) \cdot \widetilde{h}(x,b)
\end{equation}

for $x \in I_{a,b}$. Using the symmetry of $\widetilde{f}$, be obtain

\begin{equation}
    \frac{\widetilde{g}(b,x)}{\widetilde{h}(x,b)}  = \frac{\widetilde{g}(a,x)}{\widetilde{h}(x,a)} =:c^{-1}(x).
\end{equation}

\noindent
This then gives us $\widetilde{h}(x,b) = c(x)\cdot \widetilde{g}(b,x)$. We define $g(a,x):= \sqrt{c(x) \cdot c(a)} \cdot \widetilde{g}(a,x)$ and $f(a,b): = \widetilde{f}(a,b) / \sqrt{c(a) \cdot c(b)}$. Notice that $f$ and $g$ are symmetric. We can now to write $\Lambda^1$ in the form

\begin{equation}
    \Lambda^1_{a,b} (x) =  f(a,b) \cdot g(a,x) \cdot g(x,b).
\end{equation}

\noindent
Since $\Lambda_{a,b}^1$ is a probability density, $f$ must indeed have the form \eqref{eq:normalization}.

{\bf 2 implies 1.}
We define the family $\{\Lambda^{n}\}_{n \geq 1}$ recursively. Let $z_1,z_2 \in \mathbb{R}^+$ and assume that for some $m \geq 1$, the first marginals $\Lambda^{1,1},...,\Lambda^{m,1}$ have the form 

\begin{equation}
    \Lambda_{a,b}^{k,1}(x) = f_k(a,b) \cdot g_k(b,x) \cdot h_k(x,a)\cdot \mathds{1}(x \in I_{a,b})
\end{equation}

\noindent
for $a,b \in (z_1,z_2)$, and $1 \leq k \leq m$ with $f_k$ and $g_k$ symmetric. Notice that this is true for $\Lambda^1$ with $f_1(a,b) := f(a,b)$, $g_1(b,x) := g(b,x)$ and $h_1(x,a) = g(x,a)$. We define $\Lambda_{a,b}^{m+1,1}$ and $\Lambda_{a,b}^{m+1,m+1}$ as 

\begin{equation}
    \Lambda_{a,b}^{m+1,1}(x) = f_{m+1}(a,b) \cdot [f_m(b,x)]^{-1} \cdot h_m(x,a) \cdot\mathds{1}(x \in I_{a,b})
\end{equation}

and

\begin{equation}
    \Lambda_{a,b}^{m+1,m+1}(x) =  \Lambda_{b,a}^{m+1,1}(x) = f_{m+1}(b,a) \cdot [f_m(a,x)]^{-1} \cdot h_m(x,b) \cdot \mathds{1}(x \in I_{a,b})
\end{equation}

with $f_{m+1}$ the normalization

\begin{equation}
    f_{m+1}(a,b) = \Bigg(\int_{a \wedge b}^{a \vee b} [f_m(b,x)]^{-1} \cdot h_m(x,a) dx \Bigg)^{-1}.
\end{equation}

\noindent
We show that $f_{m+1}$ is symmetric. Assume $a \leq b$. Then, using the recursion,

\begin{align}
    &[f_{m+1}(a,b)]^{-1}\\
    &= \int_a^b g(a,x_1) \int_{x_1}^b   g(x_1,x_2) ... \int_{x_{m}}^b g(x_{m+1},b) dx_{m+1} ... dx_2 dx_1 \nonumber \\
    &= \int_a^b \int_{x_1}^b ... \int_{x_{m}}^b g(a,x_1) \cdot g(x_1,x_2) \cdot ... \cdot g(x_{m+1},b) dx_{m+1} ... dx_2 dx_1 \nonumber\\
    &= \int_a^b \int_a^b ... \int_a^b \mathds{1}(a \leq x_1 \leq x_2 \leq...\leq x_{m+1} \leq b) \prod_{i=0}^n  g(x_i,x_{i+1}) dx_1 dx_2 ... dx_{m+1} \nonumber
\end{align}

and

\begin{align}
    &[f_{m+1}(b,a)]^{-1}\\
    &= \int_a^b g(b,x_{m+1}) \int_a^{x_{m+1}} g(x_{m+1},x_m) ... \int_a^{x_2}\cdot g(x_1,a)  dx_1 dx_2 ... dx_{m+1} \nonumber \\
    &= \int_a^b \int_a^{x_{m+1}} ... \int_a^{x_2} g(b,x_{m+1}) \cdot g(x_{m+1},x_m) \cdot ... \cdot g(x_1,1) dx_1 dx_2 ... dx_{m+1}\nonumber\\
    &= \int_a^b \int_a^b ... \int_a^b \mathds{1}(a \leq x_1 \leq x_2 \leq...\leq x_{m+1} \leq b) \prod_{i=0}^n g(x_i,x_{i+1}) dx_1 dx_2 ... dx_{m+1} \nonumber
\end{align}

\noindent
where it is understood that $x_0 = a$ and $x_{m+2} = b$. This immediately gives the symmetry of $f_{m+1}$. We proceed to show that $\Lambda_{a,b}^{m+1,1}$ and $\Lambda_{a,b}^{m+1,m+1}$ indeed yield a density $\Lambda_{a,b}^{m+1}$ which is two sided Markov. For arbitrary $\theta_1,...,\theta_{m+1} \in \mathbb{R}^+$ we calculate

\begin{align}
    \Lambda_{a,b}^{m+1,1}(\theta_1) \cdot &\Lambda^{m}_{\theta_1,b}(\theta_2,...,\theta_{m+1}) = \Lambda_{a,b}^{m+1,1}(\theta_1) \cdot \Lambda^{m-1}_{\theta_1,\theta_{m+1}}(\theta_2,...,\theta_m) \cdot \Lambda_{\theta_1,b}^{m,m}(\theta_{m+1})\\
    =& f_{m+1}(a,b) \cdot [f_m(b,\theta_1)]^{-1} \cdot h_m(\theta_1,a) \cdot \Lambda^{m-1}_{\theta_1,\theta_{m+1}}(\theta_2,...,\theta_m) \nonumber
    \\\cdot &f_m(b,\theta_1) \cdot g_m(\theta_1,\theta_{m+1}) \cdot h_m(\theta_{m+1,b}) \cdot \mathds{1}(\theta_1 \in I_{a,b}, \theta_{m+1} \in I_{\theta_1,b}) \nonumber\\
    =&f_{m+1}(a,b) \cdot  h_m(\theta_1,a) \cdot \Lambda^{m-1}_{\theta_1,\theta_{m+1}}(\theta_2,...,\theta_m) \nonumber
    \\\cdot &g_m(\theta_1,\theta_{m+1}) \cdot h_m(\theta_{m+1},b) \cdot \mathds{1}(\theta_1 \in I_{a,b}, \theta_{m+1} \in I_{\theta_1,b}) \nonumber
\end{align}

and

\begin{align}
    \Lambda^{m}_{a,\theta_{m+1}}(\theta_1,...,\theta_{m}) \cdot &\Lambda_{a,b}^{m+1,m+1}(\theta_{m+1})  = \Lambda_{a,\theta_{m+1}}^{m,1}(\theta_1) \cdot \Lambda^{m-1}_{\theta_1,\theta_{m+1}}(\theta_2,...,\theta_m) \cdot \Lambda_{a,b}^{m+1,m+1}(\theta_{m+1})\\
    =& f_m(a,\theta_{m+1}) \cdot g_m(\theta_{m+1}, \theta_1) \cdot h_m(\theta_1,a) \cdot \Lambda^{m-1}_{\theta_1,\theta_{m+1}}(\theta_2,...,\theta_m)\nonumber \\
    \cdot &f_{m+1}(b,a) \cdot [f_m(a,\theta_{m+1})]^{-1} \cdot h_m(\theta_{m+1},b) \cdot \mathds{1}(\theta_1 \in I_{a,\theta_{m+1}}, \theta_1 \in I_{a,b}) \nonumber\\
    =&f_{m+1}(a,b) \cdot  h_m(\theta_1,a) \cdot \Lambda^{m-1}_{\theta_1,\theta_{m+1}}(\theta_2,...,\theta_m) \nonumber
    \\\cdot &g_m(\theta_1,\theta_{m+1}) \cdot h_m(\theta_{m+1},b) \cdot \mathds{1}(\theta_1 \in I_{a,b}, \theta_{m+1} \in I_{\theta_1,b}). \nonumber
\end{align}

\noindent
Observe that these two are equal. As is explained in remark \ref{rmk:construction}, one can use proposition \ref{prop:characterization_prod} to construct a $\Lambda^{m+1}$ for which the two-sided Markov property holds. Since $\Lambda^1$ is clearly symmetric, the whole family $\{\Lambda^n\}_{n\geq1}$ is symmetric as is shown in proposition \ref{prop:propagation_sym}. This concludes the proof.

\end{proof}

\noindent
We explicitly state the recursion from the proof of theorem \ref{thm:characterization_two-sided_symmetric}. This is done in the definition below, alongside the introduction of some terminology.

\begin{definition}
\label{def:generating_recursion}
    $\Lambda^1$ is called a generating density if it is of the form in \eqref{eq:form_generating} for all $z_1,z_2 \in \mathbb{R}^+$ such that $z_1 < z_2$ and $a,b \in (z_1,z_2)$. The unique symmetric two-sided Markov family $\{\Lambda^n\}_{n \geq 1}$ with first marginals

    \begin{equation}
        \Lambda_{a,b}^{n,1}(x) = f_n(a,b) \cdot g_n(b,x) \cdot h_n(x,a) \cdot \mathds{1}(x \in I_{a,b})
    \end{equation}

    \noindent
    is called the generated family. The functions $f_n,g_n$ and $h_n$ on $\mathbb{R}^+$ are defined via the recursion

    \begin{equation}
    \label{eq:recursion}
        \begin{cases}
            f_{n+1} (a,b) = \big( \int_{a \wedge b}^{a \vee b} [f_{n}(b,x)]^{-1} \cdot h_1(x,a) dx \big)^{-1} \\[2ex]
            g_{n+1} (b,x) = [f_n(a,x)]^{-1} \\[2ex]
            h_{n+1}(x,a) = h(x,a) = g(x,a)
        \end{cases}
    \end{equation}
    \noindent
    with $f_1 = f$ and $g_1 = g$. For fixed $z_1,z_2$ we call $(f_n,g_n,h_n)$ the $n$-th order factors generated by $\Lambda^{1}$. The functions $f_n$ and $g_n$ are symmetric for all $n \geq 1$, as is shown in the proof of theorem \ref{thm:characterization_two-sided_symmetric}. 
\end{definition}

\begin{remark}
    Notice that in the definition above, the $n$-th order factors are all constructed from just one function. Indeed, given only $g_1$, we know $h_1$ and $f_1$ since $h_1$ is simply defined to be equal to $g_1$ and $f_1(a,b)$ is the normalization of $g_1(b,x)\cdot h_1(x,a)$. The recursive construction gives the higher order factors.
\end{remark}


\begin{example}
\label{ex:recursion_order_stat}
\normalfont
We verify that the density of the order statistics are indeed recursively generated according to \eqref{eq:recursion}. \eqref{eq:marginals_order_stat} evaluated in $i=1$ yields

\begin{equation}
    \Lambda_{a,b}^{n,1}(x) = n \cdot |b-a|^{-n} |b-x|^{n-1}\mathds{1}(x \in I_{a,b}).
\end{equation}

 This expression is clearly factorizable for $x \in I_{a,b}$, we can take

\begin{equation}
    f_n(a,b) = n!|b-a|^{-n}, \quad g_n(b,x) = \frac{|b-x|^{n-1}}{(n-1)!}, \quad h_n(x,a) = 1.
\end{equation}

\noindent
We immediately see that $f_1(a,b) = |b-a|^{-1}$ and $g_1(b,x) = 1$ are symmetric, hence $\Lambda^1$ is indeed generating. Moreover $f_n$ and $g_n$ are symmetric and satisfy the recursion,

\begin{equation}
    f_{n+1}^{-1}(a,b) = \int_{a \wedge b}^{a \vee b} h_1(x,a) \cdot f^{-1}_n(b,x) dx = \int_{a \wedge b}^{a \vee b} \frac{|b-x|^n}{n!} dx = \frac{|b-a|^{n+1}}{(n+1)!}
\end{equation} 

and

\begin{equation}
    g_{n+1}(b,x) = [f_n(b,x)]^{-1} = \frac{|b-x|^n}{n!}.
\end{equation}

\noindent
We conclude that $f_n,g_n$ and $h_n$ are indeed the $n$-th order factors generated by $\Lambda^1$. 

\end{example}

\begin{example}[gapped order statistics]
    \normalfont
    In example \ref{ex:recursion_order_stat} we reconstructed the joint density corresponding to the order statistics of uniform random variables. In this case we had $g_1(a,b) = 1 $, defining the whole family. Since $g_1$ is indeed symmetric strictly positive, we could have directly seen that the order statistics are symmetric and two-sided Markov. In fact we recover related families if we take $g_1(a,b) = [(s-1)!]^{-1} \cdot |b-a|^{s-1}$ for any integer $s \geq 1$. For this choice of $g_1$ we obtain the so-called \emph{gapped} order statistics of the uniform random variables. These are defined as follows. For fixed $n\geq 1$ and $a,b \in \mathbb{R}^+$ we take $N = s(n+1)-1$ and consider $X_1, X_2,...,X_N$ $\sim U(a \wedge b, a \vee b)$. Then $\Theta_1 := X_{s:N}, \Theta_2 := X_{2s:N},..., \Theta_n := X_{ns:N}$ are the gapped order statistics with gap size $s$. For $a < b$, their joint densities and first marginals are known \cite{yamada2000quantitative} to be

    \begin{equation}\label{gappedjoint}
        \Lambda_{a,b}^n(\theta_1,...,\theta_n) = \frac{N!}{[(s-1)!]^n} \cdot \frac{1}{|b-a|^n} \prod_{i=1}^{n+1} (\theta_i-\theta_{i-1})^{s-1} \cdot \mathds{1}(a < \theta_1 < ... < \theta_n < b)
    \end{equation}
    
    and 
    
    \begin{equation}
        \Lambda_{a,b}^{n,1}(\theta) = \frac{(s(n+1)-1)!}{(s-1)!(sn-1)!} \cdot \frac{1}{|b-a|^{s(n+1)-1}} \cdot|b-\theta|^{sn-1} \cdot |b- \theta|^{s-1} \cdot |\theta-a|^{s-1} \mathds{1}(\theta \in I_{a,b})
    \end{equation}

    \noindent
    where it is understood that $\theta_0 = a$ and $\theta_{n+1} = b$. Akin to the density for $s=1$ in example \ref{ex:order_stat_twosidedmarkov}, the indicator is replaced by $\mathds{1}(b < \theta_n < ... < \theta_1 < a)$ in the case where $b < a$. One can check that this family is symmetric and two-sided Markov by directly verifying definition \ref{def:sym_two_sided}. We state the functions $f_n,$ $g_n$ and $h_n$ from the recursion in \eqref{eq:recursion}. The first marginals $\Lambda^{n,1}_{a,b}$ are indeed factorizable for

    \begin{equation}
        f_n(a,b) = \frac{(s(n+1)-1)!}{|b-a|^{s(n+1)-1}}, \quad g_n(b, \theta) = \frac{|b-\theta|^{sn-1}}{(sn-1)!}, \quad h_n(\theta,a) = \frac{|\theta-a|^{s-1}}{(s-1)!}.
    \end{equation}

    \noindent
    Take $n=1$. We can immediately see that $f_1,g_1$ and $h_1$ define a generating density. $f_n$ also  follows from the recursion since $g_n = f_{n-1}^{-1}$.
\end{example}

\begin{example}[Dirichlet processes]
\label{ex:Dirichlet_distribution}
    Notice that we can in fact generalize the example above by taking $s \in \mathbb{R}^+$ and defining

    \begin{equation}
        f_n(a,b) = \frac{\Gamma(s(n+1)))}{|b-a|^{s(n+1)-1}}, \quad g_n(b, \theta) = \frac{|b-\theta|^{sn-1}}{\Gamma(sn)}, \quad h_n(\theta,a) = \frac{|\theta-a|^{s-1}}{\Gamma(s)}.
    \end{equation}

    \noindent
     It is again easy to check that the recursion in \eqref{eq:recursion} works. These generalized distributions are known in literature under the name \emph{ordered Dirichlet distribution} 
     \cite{arnold2008first}. The joint density is then given by the analogue of \eqref{gappedjoint}:
     \begin{equation}\label{dirichletjoint}
        \Lambda_{a,b}^n(\theta_1,...,\theta_n) = \frac{\Gamma(s(n+1))}{[\Gamma(s)]^n} \cdot \frac{1}{|b-a|^n} \prod_{i=1}^{n+1} (\theta_i-\theta_{i-1})^{s-1} \cdot \mathds{1}(a < \theta_1 < ... < \theta_n < b)
    \end{equation}
     
\end{example}

\section{Characterizing properties of Dirichlet densities}

In this section we find characterizing properties for $\{\Lambda^n\}_{n \geq 1}$ corresponding to an ordered Dirichlet distribution, i.e., the family given in \eqref{dirichletjoint}. In Theorem \ref{thm:characterization_two-sided_symmetric}, we provided a characterization of those $\Lambda^1$ which are part of a two-sided symmetric Markov family 
$\{\Lambda^n, n\in\mathbb{N}\}$. 
Due to this characterization, it is easy to generate a large class of symmetric two-sided families. Indeed, we can perform the construction in definition \ref{def:generating_recursion} for an arbitrary integrable and symmetric function $g: \mathbb{R}^+ \times \mathbb{R}^+ \rightarrow \mathbb{R}^+$. However, we will prove that under natural assumptions such as scaling and shift invariance (detailed below), the only two-sided Markov families are the Dirichlet processes of example \ref{ex:Dirichlet_distribution}.  

\begin{definition}
\label{def:scale_shift_invar}
 The family $\{\Lambda^n\}_{n \geq 1}$ is 

 \begin{enumerate}
     \item  \textbf{Scale invariant} if for all $a,b, \gamma \in \mathbb{R}^+$, $n \geq 1$ and $\theta_1 \in I_{a,b}, \theta_2 \in I_{\theta_1,b},..., \theta_n \in I_{\theta_{n-1},b}$, 

     \begin{equation}
        \Lambda^n_{a,b}(\theta_1,...,\theta_n) = \gamma^n \cdot \Lambda_{\gamma \cdot a, \gamma \cdot b}^n (\gamma \cdot \theta_1,...,\gamma \cdot\theta_n).
    \end{equation}

    \item \textbf{Shift invariant} if, for all $a,b \in \mathbb{R}^+$, $\gamma > -(a \wedge b)$, $n \geq 1$ and $\theta_1 \in I_{a,b}, \theta_2 \in I_{\theta_1,b},..., \theta_n \in I_{\theta_{n-1},b}$

     \begin{equation}
        \Lambda^n_{a,b}(\theta_1,...,\theta_n) = \Lambda_{ a + \gamma, b + \gamma}^n (\theta_1 + \gamma,...,\theta_n + \gamma).
    \end{equation}
 \end{enumerate}

\end{definition}


We start with a useful lemma.
\begin{lemma}
\label{lemma:invariance_restriction}
Let $\phi:\mathbb{R}^2 \times \mathbb{R}^2\rightarrow \mathbb{R}$ be a continuous symmetric function which satisfies the following scale and shift properties

\begin{equation}
\begin{cases}
    \phi(\gamma \cdot x, \gamma \cdot y) = \widetilde{\lambda}(\gamma) \cdot \phi(x,y)\\
    \phi(x + \gamma, y + \gamma) = \widehat{\lambda}(\gamma) \cdot  \phi(x,y) 
\end{cases}
\end{equation}

for some continuous 
 functions $\widetilde{\lambda}, \widehat{\lambda}: \mathbb{R}^+ \rightarrow \mathbb{R}^+$. Then $\phi$ is of the form

\begin{equation}
    \phi(x,y) = t \cdot |x-y|^\sigma
\end{equation}
with $\sigma,t \in \mathbb{R}$. As a consequence,  $\widetilde{\lambda}(\gamma)=\gamma^\sigma$ and $\widehat{\lambda} = 1$.
\end{lemma}

\begin{proof}
First we will show that the functions $\widetilde{\lambda}$ and $\widehat{\lambda}$ are of the form $\widetilde{\lambda}(\gamma)=\gamma^\sigma$  and $\widehat{\lambda}(\gamma)= \exp(u \gamma)$. Later we argue that $u=0$ as in the statement of the lemma.\\

We have the following restriction on $\widetilde{\lambda}$ and $\widehat{\lambda}$. Let $\alpha, \beta \in \mathbb{R}^+$, then

    \begin{equation} \label{eq:restr_lambda_tilde}
        \widetilde{\lambda}(\alpha) \cdot \widetilde{\lambda}(\beta ) \cdot \phi(x,y) = \phi(\alpha \cdot \beta \cdot x, \alpha \cdot \beta \cdot y) = \widetilde{\lambda}(\alpha \cdot \beta) \cdot \phi(x,y)
    \end{equation}

    and 

    \begin{equation} \label{eq:restr_lambda_hat}
        \widehat{\lambda}(\alpha) \cdot \widehat{\lambda}(\beta) \cdot \phi(x,y) = \phi(\alpha + \beta + x, \alpha + \beta + y) = \widehat{\lambda}(\alpha + \beta) \cdot \phi(x,y).
    \end{equation}

    \noindent
   Let $x,y$ be such that $\phi(x,y) \neq 0$ (if such $x,y$ don't exists the lemma holds with $t=0$). We divide both sides of equations \eqref{eq:restr_lambda_tilde} and \eqref{eq:restr_lambda_hat} by $\phi(x,y)$ and obtain that

    \begin{equation}
    \label{eq:tilde_hat_restrictions}
               \widetilde{\lambda}(\alpha) \cdot \widetilde{\lambda}(\beta ) = \widetilde{\lambda}(\alpha \cdot \beta) \quad \text{and} \quad  \widehat{\lambda}(\alpha) \cdot \widehat{\lambda}(\beta) = \widehat{\lambda}(\alpha + \beta).
    \end{equation}

    \noindent
    The only functions which satisfy the restriction on $\widetilde{\lambda}$ are of the form $\widetilde{\lambda}(\gamma) = \gamma^\sigma$ with $\sigma \in \mathbb{R}$ and the only functions satisfying the restriction on $\widehat{\lambda}$ are of the form $\widehat{\lambda}(\gamma) = \exp(u \gamma)$ with $u \in \mathbb{R}$. Based on the form of $\widehat{\lambda}$ and $\widetilde{\lambda}$ we find the form of $\phi$. Assume without loss of generality that $x<y$. Notice that

    \begin{align}
        \phi(x,y) &= \widehat{\lambda}(x ) \cdot \phi(0, y-x) \\
        &= \widehat{\lambda}(x ) \cdot \widetilde{\lambda}( y - x ) \cdot \phi(0, 1 ) \nonumber \\
        &= t \cdot \exp(u x) \cdot |x-y|^\sigma, \nonumber
    \end{align}

    with $t:= \phi(0,1)$. We now verify if we indeed have the correct shifting and scaling properties for $\phi$. The scaling property states that the following two expressions should be equal for all $x,y$ and $\gamma$:

    \begin{equation}
        \phi(\gamma \cdot x, \gamma \cdot y) = \widetilde{\lambda} (\gamma) \cdot \phi(x,y) = \gamma^\sigma \cdot t \cdot \exp(u  x) \cdot   |x-y|^\sigma
    \end{equation}

    and

    \begin{equation}
        \phi(\gamma \cdot x, \gamma \cdot y) = t \cdot \exp(u \gamma x) \cdot \gamma^\sigma |x-y|^\sigma.
    \end{equation}

    This only true when $u = 0$. We now showed that $\phi(x,y)$ must have the form $\phi(x,y) = t \cdot |x-y|^\sigma$ as is stated in the lemma. It is straightforward to see that the shift property is satisfied for such functions $\phi$, regardless of what the values of $t$ and $\sigma$ are.

\end{proof}

\begin{theorem}
\label{thm:only_dirichlet}
     Let $\{\Lambda^n\}_{n \geq 1}$ be a symmetric, two-sided Markov family, i.e.,  $\Lambda^1$ is of the form \eqref{eq:form_generating}

    \begin{equation}
        \Lambda_{a,b}^1(x) = f(a,b) \cdot g(b,x) \cdot g(a,x)
    \end{equation}

        with 

    \begin{equation}
        f(a,b)^{-1} = \int_{a \wedge b}^{a \vee b} g(a,x) \cdot g(b,x) dx
    \end{equation}

    \noindent
    for some symmetric function $g:\mathbb{R}^+_0 \times \mathbb{R}^+_0 \rightarrow \mathbb{R}^+_0$. We assume $g(x,.)$ to be differentiable on $\mathbb{R}^+ \setminus \{x\}$ for all $x \in \mathbb{R}^+$.\\
    
    \noindent
    If $\{\Lambda^n\}_{n \geq 1}$ is both scale and shift invariant, then, for some $\sigma,t \in \mathbb{R}^+$,

    \begin{equation}
        g(a,x) = t \cdot |b-x|^{\sigma}.
    \end{equation}

    As a consequence we have, for $s = \sigma + 1$,

    \begin{equation}
         \Lambda_{a,b}^1(x) = f(a,b) \cdot  |b-x|^{s-1} \cdot |x-a|^{s-1}.
    \end{equation}

    \noindent
    In other words, we recover the density corresponding to the Dirichlet process with parameter $s$ from example \ref{ex:Dirichlet_distribution}.
\end{theorem}

\begin{proof}
    Notice that for $\theta_1,\theta_2 \in I_{a,b}$

    \begin{equation}
    \label{eq:start_eq}
        \frac{\Lambda^{1}_{a,b}(\theta_1)}{\Lambda^{1}_{a,b}(\theta_2)} = \frac{\Lambda^{1}_{\gamma \cdot a,\gamma \cdot b}(\gamma \cdot \theta_1)}{\Lambda^{1}_{\gamma \cdot a,\gamma \cdot b}(\gamma \cdot \theta_2)}.
    \end{equation}

    \noindent
    Bringing all the factors containing $\theta_1$ to the left hand side and all the terms containing $\theta_2$ to the right hand side shows that for any $\theta \in I_{a,b}$,

    \begin{equation}
    \label{eq:tilde_def_equation}
        \frac{g(a, \theta) \cdot g(b, \theta)}{g(\gamma \cdot a, \gamma \cdot \theta) \cdot g(\gamma \cdot b, \gamma \cdot \theta)} = \frac{\Lambda^{1}_{a,b}(\theta)}{\Lambda^{1}_{\gamma \cdot a,\gamma \cdot b}(\gamma \cdot \theta)} = \widetilde{\kappa}_1(a,b,\gamma).
    \end{equation}

    \noindent
    We define $\widetilde{g}(a,\theta):= \log [g(a,\theta)/g(\gamma \cdot a, \gamma \cdot \theta)]$ and claim that $\widetilde{g}$ does not depend on its arguments, i.e. $\widetilde{g}$ is a constant depending on $\gamma$ :

    \begin{equation} \label{eq:tilde_g_gamma_dep}
        \widetilde{g}(a,\theta) = \tilde{\kappa}_3(\gamma).
    \end{equation}

Using the definition of $\widetilde{g}$ in terms of $g$, this immediately yields the following scaling property for the function $g$: 

   \begin{equation}
        g(\gamma a, \gamma \theta) = e^{-\widetilde{\kappa}_3 (\gamma)} g(a,\theta).
   \end{equation}

    We now prove this claim. Differentiate both sides of \eqref{eq:tilde_def_equation} with respect to $\theta$ to obtain

    \begin{equation} \label{eq:after_diff}
        \widetilde{g}^{(0,1)}(a,\theta) + \widetilde{g}^{(0,1)}(b,\theta) = 0,
    \end{equation}

    \noindent
    where $\widetilde{g}^{(0,1)}(a,\cdot):= \frac{\partial g(a,t)}{\partial t}|_{t = \cdot}$ denotes the derivative of $ \widetilde{g}$ in the second entry. This expression is valid for all $\theta \in I_{a,b}$. We can deduct that

    \begin{equation} \label{eq:cases_tilde_g}
       \widetilde{g}^{(0,1)}(a,\theta) = \begin{cases}
             - \widetilde{g}^{(0,1)}(b,\theta) \quad &\text{for } \theta \in I_{a,b}\\
             +\widetilde{g}^{(0,1)}(b,\theta) \quad &\text{for } \theta \notin I_{a,b}
        \end{cases}.
    \end{equation}
 
    \noindent
    The first line of \eqref{eq:cases_tilde_g} follows directly from \eqref{eq:after_diff} and the second line can be understood as follows. Consider the case $a,b<\theta$ and let $c> \theta$. Then $\widetilde{g}^{(0,1)}(a,\theta) = -\widetilde{g}^{(0,1)}(c, \theta)$ and $\widetilde{g}^{(0,1)}(b,\theta) = -\widetilde{g}^{(0,1)}(c, \theta)$, so we indeed have $\widetilde{g}^{(0,1)}(a,\theta) = \widetilde{g}^{(0,1)}(b,\theta)$.\\

    We continue the proof of our claim under the assumption that $a < \theta$. As will become apparent later on, this is without loss of generality. For arbitrary $b_1,b_2> \theta$ we have $\theta \in I_{a,b_1}$ and $\theta \in I_{a,b_2}$. Equation \eqref{eq:cases_tilde_g} gives

    \begin{equation} \label{eq:b_1_b_2_>theta}
        -\widetilde{g}^{(0,1)}(b_1,\theta) = \widetilde{g}^{(0,1)}(a,\theta) = -\widetilde{g}^{(0,1)}(b_2,\theta).
    \end{equation}
    
    For arbitrary $b_1,b_2< \theta$ we have $\theta \notin I_{a,b_1}$ and $\theta \notin I_{a,b_2}$. In this case Equation \eqref{eq:cases_tilde_g} gives

    \begin{equation} \label{eq:b_1_b_2_<theta}
        \widetilde{g}^{(0,1)}(b_1,\theta) = \widetilde{g}^{(0,1)}(a,\theta) = \widetilde{g}^{(0,1)}(b_2,\theta).
    \end{equation}
    
   Equations \eqref{eq:b_1_b_2_<theta} and \eqref{eq:b_1_b_2_>theta} imply that $\widetilde{g}^{(0,1)}(b,\theta)$ is constant as a function of $b$ on both $(-\infty, \theta)$ and $(\theta, \infty)$. Moreover, they imply that the value on $(-\infty, \theta)$ is precisely minus the value on $(\theta, \infty)$. This means that $\widetilde{g}^{(0,1)}$ is of the form
   
   \begin{equation}
       \widetilde{g}^{(0,1)}(b,\theta) = \kappa_2(\gamma, \theta) \mathds{1}(\theta > b) - \kappa_2(\gamma,\theta) \mathds{1}(\theta < b).
   \end{equation}

   Let $\widetilde{\kappa}_2$ be the antiderivative of $\kappa_2$ with respect to $\theta$. The function $\widetilde{g}$ can then be written as

   \begin{equation} \label{eq:form_tilde_g}
       \widetilde{g}(b,\theta) = [\widetilde{\kappa}_2(\gamma, \theta) + \widetilde{\kappa}_{3,1}(\gamma)]\mathds{1}(\theta > b) + [-\widetilde{\kappa}_2(\gamma, \theta) + \widetilde{\kappa}_{3,2}(\gamma)] \mathds{1}(\theta < b),
   \end{equation}

   with $\widetilde{\kappa}_{3,1}(\gamma)$ and $\widetilde{\kappa}_{3,2}(\gamma)$ terms which only depend on $\gamma$. Notice that $\widetilde{g}$ is symmetric because, by assumption, $g$ is symmetric. Using this symmetry we obtain

   \begin{align} \label{eq:symmetry_restriction_tilde_g}
       \widetilde{g}(b, \theta) &= [\widetilde{\kappa}_2(\gamma, \theta) + \widetilde{\kappa}_{3,1}(\gamma)]\mathds{1}(\theta > b) + [-\widetilde{\kappa}_2(\gamma, \theta) + \widetilde{\kappa}_{3,2}(\gamma)] \mathds{1}(\theta < b)\\
       &= [\widetilde{\kappa}_2(\gamma, b) + \widetilde{\kappa}_{3,1}(\gamma)]\mathds{1}(\theta < b) + [-\widetilde{\kappa}_2(\gamma, b) + \widetilde{\kappa}_{3,2}(\gamma)] \mathds{1}(\theta > b) = \widetilde{g}(\theta,b).\nonumber
   \end{align}

   Equation \eqref{eq:symmetry_restriction_tilde_g} holds for all $\theta,b$ such that $\theta \neq b$. Evaluating \eqref{eq:symmetry_restriction_tilde_g} for $\theta<b$ and for $\theta>b$ yields twice the same equation which is hence valid for all $\theta,b$ such that $\theta \neq b$:

   \begin{equation}
       \widetilde{\kappa}_2(\gamma, \theta) + \widetilde{\kappa}_{3,1}(\gamma) = -\widetilde{\kappa}_2 (\gamma, b) + \widetilde{\kappa}_{3,2}(\gamma) 
   \end{equation}

   It immediately follows that $ \widetilde{\kappa}_2(\gamma, \cdot)$ is a constant only depending on $\gamma$. Hence we can define 

   \begin{equation}
       \widetilde{\kappa}_3(\gamma):= \widetilde{\kappa}_2(\gamma, \theta) + \widetilde{\kappa}_{3,1}(\gamma) = -\widetilde{\kappa}_2 (\gamma, b) + \widetilde{\kappa}_{3,2}(\gamma).
   \end{equation}

    We conclude the proof of our claim in \eqref{eq:tilde_g_gamma_dep} by writing \eqref{eq:form_tilde_g} in terms of $\widetilde{\kappa}_3$.\\

\noindent
   Notice that this scaling property for $g$ is derived from just the scale invariance. One can mutatis mutandis repeat the procedure above to derive a shift property. Indeed, in this proof we never used that $\gamma$ scales the arguments of $g$ except for the fact that this preserves the symmetry: $g(\gamma \cdot a, \gamma \cdot \theta) = g(\gamma \cdot \theta, \gamma \cdot a)$. However, the symmetry is also preserved when shifting the arguments: $g(\gamma + a, \gamma + \theta) = g(\gamma + \theta, \gamma + a)$. Therefore 
   the following scaling and shifting properties hold true,
   \begin{equation} \label{eq:shift_scale_invar_g}
   \begin{cases}
       g(a + \gamma, \theta + \gamma)= e^{-\widehat{\kappa}_3(\gamma)} \cdot g(a,\theta)\\
       g(\gamma \cdot a,\gamma \cdot \theta)= e^{-\widetilde{\kappa}_3(\gamma)} \cdot g(a,\theta)
   \end{cases} .
   \end{equation}
   
   \noindent
   Here $\widehat{\kappa}_3$ is some continuous function of $\gamma$ analogous to $\widetilde{\kappa}_3$. By virtue of lemma \ref{lemma:invariance_restriction}, we can conclude the proof.

\end{proof}

    A close inspection of the proofs of Lemma \ref{lemma:invariance_restriction} and Theorem \ref{thm:only_dirichlet} reveals that it is not enough to only assume one of the properties in \ref{def:scale_shift_invar}. That is, if one only assumes either shift or scale invariance, then there is a much wider class of symmetric two-sided Markov families. To illustrate this, we recover the specific form of the generating densities which give rise to scale, respectively shift, invariance.\\

    \noindent
    \textbf{Scale invariance:}
    Let $\{\Lambda^n\}_{n\geq1}$ be a symmetric two-sided Markov family as introduced in the statement of theorem \ref{thm:only_dirichlet} which only satisfies scale invariance. The proof of the theorem then shows that $g$ satisfies the second line of \eqref{eq:shift_scale_invar_g}, i.e.

    \begin{equation}
        g(\gamma \cdot a, \gamma \cdot \theta)= \widetilde{\lambda}(\gamma) \cdot g(a,\theta)
    \end{equation}

    \noindent
    for some continuous function $\widetilde{\lambda}(\gamma)$. The proof of lemma \ref{lemma:invariance_restriction} gives candidates for the form of $\widetilde{\lambda}$. Indeed, \eqref{eq:tilde_hat_restrictions} yields $\widetilde{\lambda}(\gamma) = \gamma^\sigma$ for some $\sigma \in \mathbb{R}$. Now we define $\varphi(x):= g(1,x) = g(x,1)$, for $x \in (0,1]$ and find

    \begin{equation}
        g(a,\theta) = (a \vee \theta)^\sigma \cdot g\Big(\frac{a}{a \vee \theta}, \frac{\theta}{a \vee \theta}\Big) = (a \vee \theta)^\sigma \cdot \varphi \Big(\frac{a \wedge \theta}{a \vee \theta}\Big).
    \end{equation}

    This yields the following generating density

    \begin{equation}
        \Lambda^1_{a,b}(\theta) = f(a,b) \cdot \theta^\sigma\cdot \varphi\Big(\frac{a\wedge b}{\theta}\Big) \varphi\Big(\frac{\theta}{a\vee b}\Big) \cdot \mathds{1}(\theta \in I_{a,b}).
    \end{equation}
    \noindent
     Here, $f(a,b)$ is the normalization of $\Lambda^1_{a,b}$. It is straightforward to check that this density is indeed scale invariant.\\

    \noindent
    \textbf{Shift invariance:}
    Let $\{\Lambda^n\}_{n\geq1}$ be a symmetric two-sided Markov family as introduced in the statement of theorem \ref{thm:only_dirichlet} which only satisfies shift invariance. The proof of theorem \ref{thm:only_dirichlet} shows that $g$ satisfies the first line of \eqref{eq:shift_scale_invar_g}, i.e.

    \begin{equation}
        g(a + \gamma, \theta + \gamma)= \widehat{\lambda}(\gamma) \cdot g(a,\theta).
    \end{equation}

    \noindent
    for some continuous function $\widehat{\lambda}(\gamma)$. The proof of lemma \ref{lemma:invariance_restriction} gives candidates for the form of $\widehat{\lambda}$. Indeed, \eqref{eq:tilde_hat_restrictions} yields $\widehat{\lambda}(\gamma) = \exp{u \gamma}$ for some $u \in \mathbb{R}$. Now we define $\varphi(x):= g(1,x) = g(x,1)$, for $x \in (0,1]$ and find

    \begin{equation}
        g(a,\theta) = e^{u \cdot (a \wedge \theta)} g(a - a \wedge \theta, \theta -a \wedge \theta) = e^{u \cdot (a \wedge \theta)} \cdot \varphi(|\theta-a|).
    \end{equation}

    This gives the following generating density

    \begin{equation}
        \Lambda^1_{a,b}(\theta) = f(a,b) \cdot e^{u \cdot \theta} \cdot \varphi(|\theta-a|) \varphi(|\theta - b|).
    \end{equation}

    \noindent
     Here, $f(a,b)$ is the normalization of $\Lambda^1_{a,b}$. It is straightforward to check that this density is indeed scale invariant.
\vspace{1cm}
\noindent
\textbf{Data Availability:}  Data sharing not applicable to this article as no datasets were generated or analysed during the current study. \\
\textbf{Declarations Competing interests:} The authors have no relevant financial or non-financial interests to disclose.

\bibliographystyle{plain}
\clearpage

\bibliography{Bibliography.bib}

\end{document}